\documentclass[11pt]{amsart}
\usepackage[utf8x]{inputenc}
\usepackage[english]{babel}
\usepackage[T1]{fontenc}
\usepackage{amssymb,amsmath}
\usepackage[hidelinks]{hyperref}
\usepackage{url}
\usepackage{indentfirst}
\usepackage{slashed}
\usepackage{enumerate,amsmath,amssymb, mathrsfs,mathtools}
\usepackage{latexsym}
\usepackage{color}
\usepackage{accents}

\usepackage[margin=2.5cm]{geometry}
\allowdisplaybreaks

\newtheorem{prop}{Proposition}[section]
\newtheorem{thm}[prop]{Theorem}
\newtheorem{lem}[prop]{Lemma}

\newtheorem{conj}[prop]{Conjecture}

\title{A Free Boundary Isometric Embedding Problem in the Unit Ball}
\author{Thomas Koerber}
\address{ University of Vienna,
	Oskar-Morgenstern-Platz 1,
	1090 Vienna,
	Austria}
\email{thomas.koerber@univie.ac.at}
\begin{document}
	\begin{abstract}
		We study a free boundary isometric embedding problem for abstract Riemannian two-manifolds with the topology of the disc. Under the assumption of positive Gauss curvature and geodesic curvature of the boundary being equal to one, we show that every such disc may be isometrically embedded into the Euclidean three-space $\mathbb{R}^3$ such that the image of the boundary meets the unit sphere $\mathbb{S}^2$ orthogonally. We also show that the embedding is unique up to rotations and reflections through planes containing the origin. Finally, we 		
		define a new Brown-York type quasi-local mass for certain free boundary surfaces and discuss its positivity.
	\end{abstract}
	\maketitle
	
	\section{Introduction}
	It is a fundamental problem in differential geometry to understand which abstract Riemannian manifolds can be realised as embedded submanifolds of a Euclidean space. In a seminal work, J. Nash showed that every sufficiently smooth Riemannian manifold can be isometrically embedded in a higher dimensional Euclidean space, see \cite{nash1956imbedding}. A similar result was later  obtained by M. Günther, see \cite{gunther1991isometric}. While these results are of broad generality, they give little information about the dimension of the ambient Euclidean space and the extrinsic geometry of the embedded manifold. \\ \indent By contrast, stronger results can be obtained in more restrictive settings. In 1916, H. Weyl conjectured that every sufficiently smooth Riemannian metric $h$ defined on the unit sphere $\mathbb{S}^2$ with positive Gauss curvature $K_h$ may be realised as a convex surface in $\mathbb{R}^3$. This problem, which is now known as the Weyl problem, was solved by H. Lewy in 1938 if $h$ is analytic, see \cite{lewy1938existence}, and in a landmark paper by L. Nirenberg if $h$ is of class $C^4$, see \cite{nirenberg1953weyl}.
	\\ \indent  As had been proposed by H. Weyl, L. Nirenberg used the continuity method in his proof. Namely, he constructed a smooth one-parameter family of positive curvature metrics $h_t,$ where $t\in[0,1],$ such that $h_1$ equals $h$ and $h_0$ is the round metric. The round metric is realised by the round sphere $\mathbb{S}^2\subset \mathbb{R}^3$ and it thus suffices to show that the set of points $t$ for which the Weyl problem can be solved is open and closed in $[0,1]$. In order to show that this set is open, L. Nirenberg used a fix point argument which is based on proving existence and estimates for a linearised equation. To show that the set is closed, he established a global $C^2$-estimate for solutions of fully non-linear equations of Monge-Ampere type. Here, the positivity of the Gauss curvature translates into ellipticity of the equation. We remark that a similar result was independently obtained by A. Pogorelov using different techniques, see \cite{pogorelov1973extrinsic}. \\ \indent
	L. Nirenberg's result has been generalized subsequently in various ways. In the degenerate case $K_h\geq 0$, P. Guan and Y. Li as well as J. Hong and C. Zuily, showed that there exists a $C^{1,1}$-embedding into the Euclidean space, see \cite{guan1994weyl,hong1995isometric}. Regarding the regularity required on the metric $h$, E. Heinz established an interor $C^2$-estimate which allowed him to relax the regularity assumption in \cite{nirenberg1953weyl} to $h$ being of class $C^3$, see \cite{heinz1959elliptic,heinz1962weyl}. Using similar techniques, F. Schulz further weakened the assumption to $h$ being of class $ C^{2,\alpha}$ for some $\alpha\in(0,1)$, see \cite{schulz2006regularity}.  We note that the isometric embedding problem for more general target manifolds, particularly with warped product metrics, has been studied. We refer to the works of P. Guan and S. Lu, see \cite{guan2017curvature}, S. Lu, see \cite{lu2016weyl}, as well as C. Li and Z. Wang, see \cite{li2016weyl}.  \\
	\indent A natural extension of the Weyl problem consists in the consideration of isometric embedding problems for manifolds with boundary. In \cite{hong1999darboux}, J. Hong considered Riemannian metrics on the disc $(D,h)$ with both positive Gauss curvature $K_h$ and positive geodesic curvature of the boundary $k_h$ and showed that $(D,h)$ can be isometrically embedded into $\mathbb{R}^3$ such that the image of the boundary is contained in a half space. In \cite{guan2007isometric}, B. Guan studied a similar embedding problem in the Minkowski space. These boundary value problems are reminiscent of the classical free boundary problem for minimal surfaces, see for instance \cite{nitsche1985stationary}. However, in the case of minimal surfaces, the variational principle forces the contact angle to be $\pi/2$. By contrast, there is no additional information in \cite{hong1999darboux} about the contact angle between the embedding of $\partial D$ and the supporting half-space. In order to make this distinction precise, we call surfaces that meet a supporting surface at a contact angle of $\pi/2$ free boundary surfaces (with respect to the supporting surface). Moreover, we call geometric boundary problems which require the solution to be a free boundary surface free boundary problems. 
	\\ \indent If the supporting surface is a half-space, free boundary problems can often be solved using a reflection argument. Consequently, it is more interesting to consider more general supporting surfaces, the most simple non-trivial example being the unit sphere. In recent years, there has been considerable activity in the study of free boundary problems with respect to the unit sphere. For example, A. Fraser and R. Schoen studied free boundary minimal surfaces in the unit ball, see \cite{fraser2011first,fraser2012minimal,fraser2016sharp}. B. Lambert and J. Scheuer studied the free boundary inverse mean curvature flow and derived a geometric inequality for convex free boundary surfaces, see \cite{lambert2016inverse,lambert2017geometric} and also \cite{volkmann2014monotonicity}. G. Wang and C. Xia proved the uniqueness of stable capillary surfaces in the unit ball, see \cite{wang2019uniqueness}, while J.  Scheuer, G. Wang and C. Xia proved certain Alexandrov-Fenchel type inequalities, see \cite{scheuer2018alexandrov}. \\ \indent 
	In this paper, we study a free boundary isometric embedding problem with respect to the unit sphere. More precisely, we prove the following theorem.
	\begin{thm}
		Let $k\geq 4$, $\alpha\in(0,1)$ and $h\in C^{k,\alpha}(\bar{D},\operatorname{Sym}(\mathbb{R}^2))$  be a Riemannian metric on the closed unit disc $\bar D$ with positive Gauss curvature $K_h$ and geodesic curvature $k_h$ along $\partial D$ equal to one. Then there exists an isometric embedding $$F:\bar{D}\to\{x\in\mathbb{R}^3:|x|\leq 1\}$$ of class $C^{k+1,\alpha}$ such that $F(\partial D)\subset \mathbb{S}^2$ and $F(\bar{D})$ meets $\mathbb{S}^2$ orthogonally along $F(\partial D)$. This embedding is unique up to rotations and reflections through planes that contain the origin.
		\label{main theorem fbie}
	\end{thm} 
	One may check that the condition $k_h=1$ is necessary. However, we expect that the regularity assumption and the condition $K_h>0$ may be weakened in a similar way as for the classical Weyl problem. Before we give an overview of the proof, we provide some motivation to study this problem besides its intrinsic geometric interest. 
	\\ \indent In general relativity, the resolution of the classical Weyl problem is used to define the so-called Brown-York mass, see \cite{brown1992quasilocal}. Let $(M,g)$ be a compact Riemannian three-manifold with boundary $\partial M$ and assume that $\partial M$ is a convex sphere. J.  Nirenberg's theorem  implies that $\partial M$ may be isometrically embedded into $\mathbb{R}^3$. We denote the mean curvature of $\partial M$ as a subset of $\mathbb{R}^3$ by $\bar H$ and the mean curvature of $\partial M$ as a subset of $M$ by $H$. The Brown York mass is then defined to be
	$$
	m_{BY}(M)=\frac{1}{8\,\pi}\int_{\partial M} (\bar H-H)\,\text{d}vol_h
	$$
	where $h$ is the metric of $\partial M$. It was already proved by S. Cohn-Vossen in 1927 that an isometric embedding of a convex surface is unique up to rigid motions provided it exists, see \cite{cohn1927zwei}. It follows that the Brown-York mass is well-defined. 
	Under the assumption that $\partial M$ is strictly mean convex and that $(M,g)$ satisfies the dominant energy condition $R\geq 0$, where $R$ denotes the scalar curvature of $(M,g)$, Y. Shi and L.-F. Tam proved in \cite{shi2002positive} that the Brown-York mass is non-negative and  that equality holds precisely if $(M,g)$ is isometric to a smooth domain in $\mathbb{R}^3$. In fact, they proved that the positive mass theorem for asymptotically flat manifolds  is equivalent to the positivity of the Brown-York mass of every compact and convex domain. A weaker inequality, which still implies the positive mass theorem, was proven by O. Hijazi and S. Montiel using spinorial methods, see \cite{hijazi2014holographic}. We would also like to mention that C.-C. Liu and S.-T. Yau introduced a quasi-local mass in the space-time case and proved positivity thereof, see \cite{liu2003positivity,liu2006positivity}. Their mass was later on generalized and further studied by M.-T. Wang and S.-T. Yau, see \cite{wang2006generalization,wang2009isometric}. Moreover, S. Lu and P. Miao derived a quasi-local mass type inequality in the Schwarzschild space, see \cite{lu2017minimal}. \\ \indent 
	In the setting of free boundary surfaces, we instead consider a compact three-manifold $(M,g)$ with a non-smooth boundary $\partial M=S\cup \Sigma$. Here, $S$ and $\Sigma$ are compact smooth surfaces meeting orthogonally along their common boundary $\partial \Sigma=\partial S$. We assume that $\partial \Sigma$ is strictly mean convex, has positive Gauss curvature and geodesic curvature along the boundary equal to one. The latter requirement is for instance satisfied if $S$ is a totally umbilical surface with mean curvature $H(S)$ equal to two. We may then isometrically embed $\Sigma$ into $\mathbb{R}^3$ with free boundary in the unit sphere. It is natural to ask if we can expect a geometric inequality in the spirit of \cite{shi2002positive} to hold on the free boundary surface $\Sigma$. In Appendix \ref{schwarzschild}, we provide some numerical evidence that this might be indeed the case. The condition $H(S)\geq 2$ in the following conjecture is natural in the context of positive mass type theorems, see for instance \cite{MR1982695}.
	\begin{conj}
		Let $(M,g)$ be a manifold with boundary $\partial M=S\cup \Sigma$ where $S$ and $\Sigma$ are smooth discs meeting orthogonally along $\partial \Sigma=\partial S$. Assume that $\Sigma$ is strictly mean convex, has positive Gauss curvature as well as geodesic curvature along $\partial \Sigma$ equal to one and that $R\geq 0$ on $M$ and $H(S)\geq 2$ on $S$. Then there holds
		$$
		\int_{\Sigma} (\bar H-H)\,\text{d}vol_{h}\geq 0.
		$$
		Equality holds if and only if $(M,g)$ is isometric to a domain in
		\label{conjecture} $\mathbb{R}^3$ and if the isometry maps $S$ to a subset of $\mathbb{S}^2$.
	\end{conj}
	A positive answer to this conjecture would mean that in certain cases, mass can be detected with only incomplete information about the geometry of the boundary of the domain under consideration. It turns out that proving such an inequality is not straightforward. On the one hand, there does not seem to be an obvious approach to adapt spinorial methods such as in \cite{hijazi2014holographic} to the free boundary setting. On the other hand, arguing as in \cite{shi2002positive}, we observe that a positive answer to the conjecture is equivalent to a positive mass type theorem for asymptotically flat manifolds which are modeled on a solid cone.
	\\ \indent 
	We now describe the proof of Theorem \ref{main theorem fbie}. As in \cite{nirenberg1953weyl}, we use the continuity method and smoothly connect the metric $h$ to the metric $h_0$ of the spherical cap whose boundary has azimuthal angle $\pi/4$. In particular, the free boundary isometric embedding problem can be solved for $h_0$. We then need to show that the solution space is open and closed. Contrary to the argument by Nirenberg, the non-linearity of the boundary condition does not allow us to use a fixed point argument. We instead use a power series of maps, as H. Weyl had initially proposed, where the maps are obtained as solutions of a linearised first-order system. This system was previously  studied by C. Li and Z. Wang in \cite{li2016weyl}. Another difficulty arises from the fact that the prescribed contact angle makes the problem seemingly overdetermined. We thus solve the linearised problem without prescribing the contact angle and recover the additional boundary condition from the constancy of the geodesic curvature and a lengthy algebraic computation. Regarding the convergence of the power series, we prove a-priori estimates using the Nirenberg trick. Here, it turns out that a recurrence relation for the Catalan numbers plays an important part. In order to show that the solution space is closed, we observe that the Codazzi equations imply a certain algebraic structure for the normal derivative of the mean curvature at the boundary. This allows us to use the maximum principle  to prove a global $C^2$-estimate for every isometric embedding. \\ \indent 
	The rest of this paper is organized as follows. In Section 2, we define the so-called solution space and show that the space of metrics on $D$ with positive Gauss curvature and geodesic curvature along $\partial D$ being equal to $1$ is path-connected. In Section 3, we study the linearised problem and show that the solution space is open. In Section 4, we prove a global curvature estimate and show that the solution space is closed. In Section 5, we prove Theorem \ref{main theorem fbie}. 
	\\\indent\textbf{Acknowledgements.} The author would like to thank his PhD advisor Guofang Wang for suggesting the problem and for many helpful conversations.
	\section{Basic properties of the solution space}
	Let $\alpha\in(0,1)$ and fix an integer $k\geq  2$. Let $D\subset\mathbb{R}^2$ be the unit two-disc and consider a $C^{k,\alpha}$-Riemannian metric $h$ defined on the closure $\bar{D}$. We assume that the Gauss curvature $K_h$ is positive and that the geodesic curvature of $\partial D$ satisfies $k_h=1$. \\ \indent We would like to find a map  
	$ F\in C^{k+1,\alpha}(\bar{D},\mathbb{R}^3)$ which isometrically embeds the Riemannian manifold $(\bar{D},g)$ into $\mathbb{R}^3$ such that the boundary $ F(\partial D)$ meets the unit sphere $\mathbb{S}^2$  orthogonally. To this end, we consider the space $\mathcal{G}^{k,\alpha}$  of all Riemannian metrics $\tilde h\in C^{k,\alpha}(\bar{D},\operatorname{Sym}(\mathbb{R}^2)$ with positive Gauss curvature and geodesic curvature along $\partial D$ equal to one. Up to a diffeomorphism of $\bar{D}$, we may write the metric $\tilde h$ in isothermal coordinates, see \cite{deturck1981some}. This means that there exists a positive function $\tilde E\in C^{k,\alpha}(\bar{D})$ such that
	\begin{align}
	\tilde h=\tilde E^2(dx_1^2+dx_2^2)=\tilde E^2(dr^2+r^2d\varphi^2). \notag
	\end{align}
	Here, $(x_1,x_2)$ denote the standard Euclidean coordinates on $D$ while $(\varphi,r)$ denote polar coordinates centred at the origin. A map $\tilde F$ isometrically embeds $(\bar{D},\tilde h)$ into $\mathbb{R}^3$ with free boundary in the unit sphere if and only if it solves the following boundary value problem
	\begin{align}
	\begin{dcases}
	& d \tilde F\cdot d  \tilde F=\tilde h \quad \text{ in } D, \\
	& |\tilde F|^2=1  \quad  \quad\hspace{.25cm} \text{ on } \partial D,
	\\
	&  \partial_r  \tilde F=  \tilde E  \tilde F \quad \hspace{.25cm}  \text{ on } \partial D.
	\end{dcases} 
	\label{fbp}
	\end{align}  
	\indent 
	We will prove the existence of such a map $F$ corresponding to the metric $h$ using the continuity method. More precisely, we define the solution space $\mathcal{G}_*^{k,\alpha}\subset \mathcal{G}^{k,\alpha}$ to be the space which contains all metrics $\tilde h\in\mathcal{G}^{k,\alpha}$ for which there exists a map $\tilde F\in C^{k+1,\alpha}(\bar{D},\mathbb{R}^3)$ solving the problem (\ref{fbp}). One may construct explicit examples to see that the space $\mathcal{G}_*^{k,\alpha}$ is non-empty. In order to show that $\mathcal{G}_*^{k,\alpha}=\mathcal{G}^{k,\alpha}$,  we equip $\mathcal{G}^{k,\alpha}$ with the $C^{k,\alpha}(\bar{D},\operatorname{Sym}(\mathbb{R}^2))$-topology and show that $\mathcal{G}^{k,\alpha}$ is path-connected, open and closed. It actually turns out that $\mathcal{G}^{k,\alpha}$ is not only path-connected but that the paths can be chosen to be analytic maps from the unit interval to $\mathcal{G}^{k,\alpha}$. \\ \indent In the proof of the next lemma, the connection and Laplacian of the Euclidean metric $\bar g$ on $D$ will be denoted by $\bar \nabla$ and $\bar \Delta$, respectively.
	\begin{lem}
		Given $h_0,$ $h_1\in\mathcal{G}^{k,\alpha}$, there exists an analytic map 
		$$
		h:[0,1]\to \mathcal{G}^{k,\alpha}
		$$
		such that $h(0)=h_0$ and $h(1)=h_1$. In particular, the space $\mathcal G^{k,\alpha}$ is path-connected. 
		\label{pathconnect}
	\end{lem}
	\begin{proof}
		Let $h_0,$ $h_1\in\mathcal{G}^{k,\alpha}$ and choose isothermal coordinates with conformal factors $E_0,E_1$, respectively. Given $t\in[0,1]$, we define a one-parameter family of Riemannian metrics $h_t$ connecting $h_0$ and $h_1$ to be  
		\begin{align}
		h_t= E_t^2(dx_1^2+dx_2^2)= E_t^2(dr^2+r^2d\varphi^2) 
		\end{align}
		where \begin{align}
		E_t=\frac{E_0 E_1}{(1-t) E_1+ tE_0}	 
		\end{align}	
		denotes the conformal factor of the metric $h_t$. Since both $E_0$ and $E_1$ are positive, $E_t$ is well-defined. A straightforward computation reveals that given a metric $h$, the Gauss curvature and the geodesic curvature of $\partial D$ satisfy the formulae
		\begin{align}
		k_{h}=-r\partial_r((rE)^{-1})\big|_{r=1}, \qquad K_h=-E^{-2}\bar\Delta\log E, 
		\end{align}
		see \cite{han2006isometric}. This implies that
		\begin{align*}
		k_{h_t}=-r\partial_r\big((1-t)(rE_0)^{-1}+t(r E_1)^{-1}\big)\big|_{r=1}=1 
		\end{align*}
		for every $t\in[0,1]$. In order to see that the Gauss curvature is positive, we abbreviate $K_0=K_{h_0}$, $K_1=K_{h_1}$ and compute
		\begin{align*}
		&\bar\Delta \log E_t\\=&\bar\Delta \log E_0 + \bar\Delta \log E_1 -\bar\Delta \log(tE_0+(1-t) E_1)\\
		=& -K_0E_0^2-K_1E_1^2 -\frac{t\bar\Delta E_0+(1-t)\bar\Delta E_1}{tE_0+(1-t)E_1}+\frac{|t\bar\nabla E_0+(1-t)\bar\nabla  E_1|^2}{(tE_0+(1-t)E_1)^2}
		\\=& -K_0E_0^2-K_1 E_1^2
		+(tE_0+(1-t) E_1)^{-1}\bigg(tE_0^3K_0+(1-t) E_1^3K_1-t\frac{|\bar\nabla E_0|^2}{E_0}-(1-t)\frac{|\bar\nabla  E_1|^2}{ E_1}\bigg)
		\\&+(tE_0+(1-t)E_1)^{-2}\big(t^2|\bar\nabla E_0|^2+(1-t)^2|\bar\nabla  E_1|^2+2t(1-t)|\bar\nabla E_0||\bar\nabla E_1|\big)
		\\=&-(tE_0+(1-t)E_1)^{-1}\big((1-t) E_1 E_0^2K_0+t E_0E_1^2K_1\big) 
		\\&-(tE_0+(1-t) E_1)^{-2}\big(t(1-t)E_1E_0^{-1}|\bar\nabla E_0|^2+t(1-t)E_0E_1^{-1}|\bar\nabla E_1|^2-2t(1-t)|\bar\nabla E_0||\bar\nabla  E_1|\big)
		\\<&0.
		\end{align*}
		In the last inequality, we used Young's inequality and the  positivity of $K_0 $, $K_1,$ $E_0$ and $E_1$ as well as $t\in[0,1]$. In particular, $K_{h_t}$ is positive for every $t\in[0,1]$. Clearly, $h_t$ is analytic with respect to $t$. 
	\end{proof}
	
	\section{Openness of the solution space}
	In this section, we show that the solution space $\mathcal{G}_*^{k,\alpha}$ is open.  Let $h\in\mathcal{G}^{k,\alpha}$ and suppose that there is a map $F\in C^{k+1,\alpha}(\bar{D},\mathbb{R}^3)$ satisfying (\ref{fbp}) with respect to the metric $h$. Since the Gaussian curvature $K=K_h$ is strictly positive, it follows that $F(\bar{D})$ is strictly convex. We will now show that for every $\tilde h\in\mathcal{G}^{k,\alpha}$ which is sufficiently close to $h$ in the $C^{2,\alpha}(\bar{D},\operatorname{Sym}(\mathbb{R}^2))$-topology for some $\alpha\in(0,1)$ there exists a solution $\tilde F\in C^{3,\alpha}(\bar{D},\mathbb{R}^3)$ of (\ref{fbp}).  To this end, we use Lemma \ref{pathconnect} to find a path $h_t$ connecting $h$ and $\tilde h$ in $\mathcal{G}^{k,\alpha}$ and observe that $h_t$ and all of its spacial derivatives are analytic with respect to $t$.
	\subsection{The linearised problem}
	
	We define a family of $C^{3,\alpha}(\bar{D},\mathbb{R}^3)$-maps $$F_t:[0,1]\times \bar{D}\to\mathbb{R}^3$$ which is analytic with respect to $t$ and satisfies $F_0=F$. This means that
	\begin{align}
	F_t=\sum_{l=0}^\infty \Psi^l\frac{t^l}{l!}, \label{taylor}
	\end{align}
	where $\Psi^0=F_0$ and $\Psi^l\in C^{3,\alpha}(\bar{D},\mathbb{R}^3)$ are solutions to certain linearised equations if $l\geq1$. Clearly, if (\ref{fbp}) is solvable up to infinite order at $t=0$ and if $F_t$ and all its time derivatives converge in $C^{2,\alpha}(\bar{{D}},\mathbb{R}^3)$ for every $t\in[0,1]$, then it follows  that $\tilde F=F_1$ is a solution of (\ref{fbp}) with respect to $\tilde h$. More precisely, we have the following lemma.
	\begin{lem}
		Consider two metrics $h$, $\tilde h\in\mathcal{G}^{k,\alpha}$ with conformal factors $E$, $\tilde E$ and let $h_t:[0,1]\to\mathcal{G}^{k,\alpha}$ be the connecting path from Lemma \ref{pathconnect} with conformal factor $E_t$. Let $F$ be a solution of the free boundary problem (\ref{fbp}) with respect to $h$ and suppose that $F_t$ as defined in (\ref{taylor}) converges in $C^{2,\alpha}(\bar{D})$  uniformly. Then the following identities hold
		\begin{align*}
		\frac{d^l}{dt^l}h_t\bigg|_{t=0}&=\operatorname{Id}\sum_{i=0}^l\binom{l}{i}\partial^i_tE_t\partial_t^{l-i}E_t \bigg|_{t=0},  \\
		\frac{d^l}{dt^l}(dF_t\cdot dF_t)\bigg|_{t=0}&=2dF\cdot d\Psi^l+\sum_{i=1}^{l-1}\binom{l}{i}d\Psi^{l-i}\cdot d\Psi^{i},  \\
		\frac{d^l}{dt^l}|F_t|^2\bigg|_{t=0}&=2F\cdot \Psi^l+ \sum_{i=1}^{l-1}\binom{l}{i}\Psi^{l-i}\cdot \Psi^{i}, \\
		\frac{d^l}{dt^l} \partial_rF_t\bigg|_{t=0}&=\partial_r \Psi^l, \\
		\frac{d^l}{dt^l}( E_tF_t)\bigg|_{t=0}&=\sum_{i=0}^{l}\binom{l}{i}\partial_t^iE_t\Psi^{l-i}\bigg|_{t=0}.
		\end{align*} 	
	\end{lem}
	We therefore consider the following boundary value problem. For every $l\in\mathbb{N}$, we seek to find a map $\Psi^l$ satisfying
	\begin{align}
	\begin{dcases}
	& dF\cdot d\Psi^l+\frac12\sum_{i=1}^{l-1}\binom{l}{i}d\Psi^{l-i}\cdot d\Psi^{i} =\frac12\operatorname{Id}\sum_{i=0}^l\binom{l}{i}\partial_t^iE_t\partial_t^{l-i}E_t \big|_{t=0} \qquad \text{ in } D,\\
	& F\cdot \Psi^l+\frac12 \sum_{i=1}^{l-1}\binom{l}{i}\Psi^{l-i}\cdot \Psi^{i}=0  \hspace{0.35cm} \text{ on } \partial D, \\
	& \partial_r \Psi^l =\sum_{i=0}^{l}\binom{l}{i}\partial_t^iE_t\Psi^{l-i}\big|_{t=0} \hspace{0.1cm}\qquad \text{ on } \partial D.\end{dcases}	 
	\label{linearizedequations}
	\end{align}  
	The first equation is understood in the sense of symmetric two-tensors, which means that we use the convention $d^1d^2=d^2d^1=\operatorname{Sym}(d^1\otimes d^2)$, where $d^1,d^2$ are the dual one-forms of the given coordinate system. 
	\begin{lem}
		Consider two metrics $h,$ $\tilde h\in\mathcal{G}^{k,\alpha}$ with conformal factors $E,$ $\tilde E$ and let $h_t:[0,1]\to\mathcal{G}^{k,\alpha}$ be the connecting path from Lemma \ref{pathconnect}. Let $F$ be a solution of the free boundary problem (\ref{fbp}) with respect to $h$ and suppose that there exists a family $\{\Psi^l\in C^{2,\alpha}(\bar{D}):l\in\mathbb{N}\}$ satisfying (\ref{linearizedequations}) for every $l\in\mathbb{N}$. Furthermore, assume that $F_t$ as defined in (\ref{taylor}) converges in $C^{2,\alpha}(\bar{D})$  uniformly. Then there holds
		$$
		\begin{dcases}
		& d F_t\cdot d  F_t=h_t \quad \text{ in } D, \\
		& |F_t|^2=1 \quad \hspace{0.8cm} \text{ on } \partial D, \\
		&  \partial_r   F_t=   E_tF_t \quad \hspace{0.25cm} \text{ on } \partial D.
		\end{dcases}
		$$
		for every $t\in[0,1]$. Here, $E_t$ is the conformal factor of the metric $h_t$.
		\label{existence criterion}
	\end{lem}
	\begin{proof}
		This follows from the previous lemma and well-known facts about analytic functions.
	\end{proof}
	
	We now modify the approach used in  \cite{li2016weyl} to find solutions of (\ref{linearized equations}).	
	For ease of notation, we fix an integer $l\in\mathbb{N}$ and define $\Psi=\Psi^l$. Given a coordinate frame $\partial_1,$ $\partial_2$ with dual one-forms $d^1,$ $d^2$,   we   define $$u_i=\Psi\cdot \partial_iF$$ and $$v=\Psi \cdot \nu,$$ where $\nu$ is the unit normal of the surface $F(\bar{D})$ pointing inside the larger component of $B_1(0)-F(D)$. Here and in the following, the Latin indices $i,j,m,n$ always range over $\{1,2\}$ unless specified otherwise.  We  consider the one-form
	\begin{align*}
	w=\Psi\cdot dF=u_id^i.	 
	\end{align*}
	We denote the second fundamental form of the surface $F$ with respect to our choice of unit normal $\nu$ by $A=(A_{ij})_{ij}$ and compute
	\begin{align*}
	\nabla_h w&= \partial_ju_id^i d^j -\Gamma^m_{i,j}u_md^id^j\\
	&=\partial_iF\cdot\partial_j\Psi d^i d^j +\Gamma^m_{i,j}\Psi\cdot \partial_mF d^i d^j -\Psi\cdot \nu A_{ij}d^i d^j-\Gamma^m_{i,j}u_md^i d^j \\
	&=\partial_i F\cdot\partial_j\Psi d^i d^j-v A_{ij}d^i d^j,
	\end{align*} 
	where $\nabla_h$ and $\Gamma^m_{i,j}$ denote the Levi-Civita connection and  Christoffel-Symbols of the metric $h$, respectively. In the second equation, we used that $$\partial_i\partial_j F= \Gamma^{m}_{i,j}\partial_m F-A_{ij}\nu.$$ Hence, if $\tilde q \in C^{k,\alpha}(\bar{D},\operatorname{Sym}(\mathbb{R}^2))$ is a symmetric covariant two-tensor, $\psi\in C^{k+1,\alpha}(\bar{D})$ and $\Phi\in C^{k,\alpha}(\bar{D},\mathbb{R}^3)$,  the system 
	\begin{align*}
	\begin{dcases}
	& dF\cdot d\Psi=\tilde q \hspace{0.05cm}\text{ in } D,\\
	& F\cdot \Psi=\psi \hspace{0.05cm}\text{ on } \partial D, \\
	& \partial_r \Psi =\Phi \hspace{0.05cm}\text{ on } \partial D,\end{dcases}	 
	\end{align*}
	is equivalent to 
	\begin{align}
	\begin{dcases}
	&\partial_1u_{1}-\Gamma^m_{1,1}u_m=\tilde q_{11}-v A_{11} \text{ in }D,\\
	&\partial_1u_{2}+\partial_2u_{1}-2\Gamma^m_{1,2}u_m=2(\tilde q_{12}-v A_{12})\text{ in }D,\\
	&\partial_2u_{2}-\Gamma^m_{2,2}u_m=\tilde q_{22}-v A_{22}\text{ in }D,\\
	& w(\partial_r)=E\psi \text{ on } \partial D, \\
	& \partial_r \Psi =\Phi \hspace{0.0cm}\text{ on } \partial D.\end{dcases}	 
	\label{linearized equations}
	\end{align}
	This system is over-determined. However, it turns out that it suffices to solve the first four lines, as in the special situation of (\ref{linearizedequations}), the last equation will be automatically implied by the constancy of the geodesic curvature along $\partial D$. Another way of writing the first three lines of (\ref{linearized equations}) is $$\text{Sym}(\nabla_h w)=\tilde q-v A.$$ Since $F(\bar{D})$ is strictly convex, the second fundamental form $A$ defines a Riemannian metric on $\bar{D}$. We can thus take the trace with respect to $A$ to find
	\begin{align*}
	v = \frac12\operatorname{tr}_A(\tilde q-\text{Sym}(\nabla_h w) ).
	\end{align*}
	Given $b\in\{0,1\}$, we define $$\mathcal{A}^{b,\alpha}_A\subset C^{b,\alpha}(\bar{D},\operatorname{Sym}(T^*\bar{D}\otimes T^*\bar{D}))$$ to be the sections of the bundle of trace-free (with respect to $A$) symmetric (0,2)-tensors of class $C^{b,\alpha}$. Moreover, we define $$\mathcal{B}^{b,\alpha}_A=C^{b,\alpha}(\bar{{D}},T^*\bar{D})$$ to be the one-forms of class $C^{b,\alpha}$. Next, we define the operator $${L}:\mathcal B^{1,\alpha}_A\to\mathcal A^{0,\alpha}_A, \quad
	\omega\mapsto\operatorname{Sym}(\nabla_h \omega)-\frac12\operatorname{tr}_A(\operatorname{Sym}(\nabla_h \omega))A.
	$$
	We are then left to find a solution $w$ of the equation $${L}(w)=\tilde q-\frac12 \operatorname{tr}_A(\tilde q)A.$$ We abbreviate the right hand side by $q$. We note that $\mathcal A^{b,\alpha}_A$ and $\mathcal{B}^{b,\alpha}_A$ are both isomorphic to $\mathcal{C}^{b,\alpha}=C^{b,\alpha}(\bar{D},\mathbb{R}^2)$ via the isomorphisms
	\begin{align*}
	&(a_1,a_2)\mapsto a_1dx^1dx^1+a_2(dx^1dx^2+dx^2dx^1)-(A^{22})^{-1}(A^{11}a_1+2A^{12}a_2)dx^2 dx^2,\\
	&(u_1,u_2)\mapsto u_idx^i.
	\end{align*}
	Here, $x_1,x_2$ are the standard Euclidean coordinates. At this point, we emphasize that $(A^{ij})_{ij}$ denotes the {inverse tensor} of the second fundamental form $A$  and {not} the quantity $h^{im}h^{jn}A_{mn}$. We also define inner products on $\mathcal A^{b,\alpha}_{A}$ and $\mathcal{B}^{b,\alpha}_{A}$, respectively. Namely,
	\begin{align*}
	&\langle q^0,q^1\rangle_{\mathcal{A}_A}=\int_D KA^{ij}A^{mn}q^0_{im}q^1_{jn}\,\text{d}{vol}_h=\int_D K\langle q^0,q^1\rangle_A\,\text{d}{vol}_h,\\
	&\langle \omega^0,\omega^1\rangle_{\mathcal{B}_A}=\int_D KA^{ij}\omega^0_i\omega^1_j\,\text{d}{vol}_h=\int_D K\langle \omega^0,\omega^1\rangle_A\,\text{d}vol_h,
	\end{align*}
	where $q^0,$ $q^1\in\mathcal{A}^{b,\alpha}_A$ and $\omega^0,$ $\omega^1\in \mathcal{B}^{b,\alpha}_A$. The particular choice of the inner product does not really matter as the strict convexity implies that these inner products are equivalent to the standard $L^2$-product with respect to the metric $h$. However, our choice implies a convenient form for the adjoint operator of ${L}$.\\ \indent 
	As was observed by C. Li and Z. Wang in \cite{li2016weyl}, the operator ${L}$ is elliptic. For the convenience of the reader, we provide a brief proof.
	\begin{lem}
		The operator $L$ is a linear elliptic first-order operator.
		\label{ellipticity}
	\end{lem}
	\begin{proof}
		We regard $L$ as an operator from $\mathcal{C}^{1,\alpha}$ to $\mathcal{C}^{0,\alpha}$. The leading order part of the operator is given by
		\begin{align}
		\tilde L(u_1,u_2)=(\partial_1 u_1-\frac{A_{11}}{2}A^{ij}\partial_i u_j,\frac{\partial_1 u_2+\partial_2 u_1}{2}-\frac{A_{12}}{2}A^{ij}\partial_i u_j).
		\end{align}
		In order to compute the principal symbol at a point $p\in D$, we may rotate the coordinate system such that $A$ is diagonal. It then follows that
		\begin{align}
		\tilde L(u_1,u_2)=\begin{pmatrix}
		\frac12 & 0 \\ 0 & \frac12
		\end{pmatrix}
		\partial_1 \begin{pmatrix}
		u_1 \\ u_2
		\end{pmatrix} + \begin{pmatrix}
		0 & -\frac12\frac{A_{11}}{A_{22}} \\ \frac12 & 0. 
		\end{pmatrix}\partial_2 \begin{pmatrix}
		u_1 \\ u_2
		\end{pmatrix}.
		\end{align}
		Given $\xi \in \mathbb{S}^1$, the principal symbol  thus equals
		$$
		\sigma_L(\xi)=\sigma_{\tilde L}(\xi)=\frac12
		\begin{pmatrix}
		\xi_1 & -\frac{A_{11}}{A_{22}}\xi_2  \\
		\xi_2 & \xi_1
		\end{pmatrix}.
		$$
		Consequently, $$4\det(\sigma_L(\xi))=\xi^2_1+\xi_2^2 A^2_{11}/A^2_{22}>0,$$ as $F(D)$ is strictly convex. 
	\end{proof}
	We proceed to calculate the adjoint of ${L}$ denoted by ${L}^*$. Let $\omega\in \mathcal B^{1,\alpha}_A$ and $q\in \mathcal A^{1,\alpha}_A$. We denote the outward co-normal of $\partial D$ by $\mu=\mu^i\partial_i$ and the formal adjoint of $\nabla_h$, regarded as an operator mapping $(0,1)$-tensors to $(0,2)$-tensors, by $-\operatorname{div}_h$, regarded as an operator mapping $(0,2)$-tensors to $(0,1)$-tensors. Moreover, we denote the musical isomorphisms of $h$ and $A$ by $\sharp_A, $ $\sharp_h, $ $\flat_A$ and $\flat_h$.  Using integration by parts and the fact that $q$ is trace free with respect to $A$, we obtain
	\begin{equation}
	\label{adjoint equation}
	\begin{aligned}
	\langle L(\omega),q\rangle_{\mathcal{A}_A}=&\int_D K\langle \nabla_h \omega,q\rangle_A \,\text{d}{vol}_h-\frac12\int_D KA^{ij}A^{mn} \operatorname{tr}_A(\text{Sym}(\nabla_h \omega))A_{im}q_{jn}\,\text{d}{vol}_h 
	\\=&\int_D K\langle \nabla_h \omega, q^{\sharp_A\sharp_A\flat_h\flat_h}\rangle_h\, \text{d}{vol}_h
	\\=&-\int_D \langle \omega,\operatorname{div}_h(K q^{\sharp_A\sharp_A\flat_h\flat_h})\rangle_h\,\text{d}{vol}_h+\int_{\partial D} K \langle \omega,\iota_\mu q^{\sharp_A\sharp_A\flat_h\flat_h}\rangle_h \,\text{d}{vol}_h
	\\=&-\int_D K\langle \omega,\frac1K(\operatorname{div}_h(Kq^{\sharp_A\sharp_A\flat_h\flat_h}))^{\flat_A\sharp_h}\rangle_A\,\text{d}{vol}_h+\int_{\partial D} K\langle \omega,\iota_{\mu^{\flat_h}}q^{\sharp_A} \rangle_A\, \text{d}{vol}_h.
	\end{aligned}
	\end{equation}
	Consequently, the adjoint operator $$L^*:\mathcal{A}^{1,\alpha}_A\to\mathcal{B}^{0,\alpha}_A$$ is given by 
	\begin{align}
	q\mapsto-\frac1K(\operatorname{div}_h(Kq^{\sharp_A\sharp_A\flat_h\flat_h}))^{\flat_A\sharp_h}. \label{adjoint operator}
	\end{align}
	As ${L}^*$ is the adjoint of an elliptic operator, it is elliptic itself. We now take a closer look at the boundary term. We consider isothermal polar coordinates $(\varphi,r)$ centred at the origin with conformal factor $E$  and compute at $\partial D$, using the free boundary condition, that 
	\begin{align}
	-A_{r\varphi}=\partial_\varphi\partial_r F\cdot \nu=\partial_\varphi (EF)\cdot \nu=0. \label{A cond boundary}
	\end{align}
	The last equality holds since $F,$ $\partial_\varphi F$ are both tangential at $\partial D$. Consequently, we have $A^{r\varphi}\equiv 0$ on $\partial D$. Combining this with $\operatorname{tr}_A(q)=0$ and writing $\omega=\omega_id^i$ we obtain 
	\begin{equation}
	\begin{aligned}
	\int_{\partial D} K\langle \omega,\iota_{\mu^{\flat_h}}q^{\sharp_A} \rangle_A \,\text{d}{vol}_h&= 
	\int_{\partial{D}}KA^{ij}A^{mn}\omega_i\mu^{\flat_h}_mq_{jn}\,\text{d}{vol}_h
	\\&=\int_{\partial D} KA^{rr}\mu^{\flat_g}_r(A^{\varphi\varphi}\omega_\varphi q_{\varphi r}+A^{rr}\omega_rq_{rr})\,\text{d}{vol}_h
	\\&=\int_{\partial D}^{} KA^{rr}A^{\varphi\varphi}\mu^{\flat_h}_r(\omega_\varphi q_{\varphi r}-\omega_rq_{\varphi\varphi})\,\text{d}{vol}_h
	\\&=\int_{\partial D}^{} \frac{1}{E^2}\mu^r(\omega_\varphi q_{\varphi r}-\omega_rq_{\varphi\varphi})\,\text{d}{vol}_h
	\\&=\int_{\partial D}  (\langle \omega,\iota_\mu q\rangle_{(\partial D,h)}-\omega(\mu)\operatorname{trace}_{(\partial D,h)}(q))\,\text{d}{vol}_h.	
	\end{aligned}
	\label{boundary integral}
	\end{equation}
	This leads us to define the boundary operators 
	$$R_1:\mathcal{B}^{0,\alpha}_A\to C^{0,\alpha}(\partial D), \qquad R_1^*:\mathcal{A}^{0,\alpha}_A\to C^{0,\alpha}(\partial D),$$ $$R_2:\mathcal{B}^{0,\alpha}_A\to C^{0,\alpha}(\partial D,T^*\partial D), \qquad R^*_2:\mathcal{A}^{0,\alpha}_A\to C^{0,\alpha}(\partial D, T^*\partial D)
	$$
	to be 
	\begin{align*}
	& R_1(\omega)=\omega(\mu),\qquad\qquad\hspace{0.4cm} R^*_1(q)=\operatorname{trace}_{(\partial D,h)}(q),  \\
	&R_2(\omega)=e^*\omega, \qquad\qquad\hspace{0.2cm} \quad  R^*_2(q)=e^*\iota_\mu q.
	\end{align*}
	Here, $e:\partial D\to\bar{{D}}$ denotes the inclusion map. Clearly, given $b\in\{0,1\}$, the spaces $C^{b,\alpha}(\partial D)$ and $C^{b,\alpha}(\partial D,T^*\partial D)$ are  isomorphic. \\ \indent 
	Next, we define the operators $$\mathcal{L}:\mathcal{B}^{1,\alpha}_A\to\mathcal{A}^{0,\alpha}_A\times C^{1,\alpha}(\partial D), \quad \omega\mapsto (L\omega,R_1(\omega))$$ and $$\mathcal{L}^*:\mathcal{A}^{1,\alpha}_A\to \mathcal{B}^{0,\alpha}_A\times C^{1,\alpha}(\partial D, T^*\partial D), \quad q\mapsto (L^*(q),R_2^*(q)).$$ Given $\omega^0,$ $\omega^1\in \mathcal{B}^{0,\alpha}_A$, $q^0,$ $q^1\in\mathcal{A}^{0,\alpha}_A$, $\psi\in C^{0,\alpha}(\partial D)$ and $\zeta \in C^{0,\alpha}(\partial D, T^*\partial D)$, we define the inner products 
	\begin{align*}
	\langle\langle (q^0,\psi),q^1\rangle\rangle_{\mathcal{A}_A} &= \langle q^0,q^1 \rangle_{\mathcal{A}_A}+\int_{\partial D} \psi R^*_1(q^1)\,\text{d}{vol}_h\\
	\langle\langle \omega^0,(\omega^1,\zeta)\rangle\rangle_{\mathcal{B}_A}&=\langle \omega^0,\omega^1 \rangle_{\mathcal{B}_A} +\int_{\partial D}\langle\zeta,  R_2(\omega^0)\rangle_{(\partial D,h)}\,  \text{d}{vol}_h.
	\end{align*}
	According to (\ref{adjoint equation}) and (\ref{boundary integral}), there holds
	\begin{align}
	\langle \langle \mathcal{L}(\omega),q\rangle\rangle_{\mathcal{A}_A}=\langle\langle \omega,\mathcal{L}^*(q)\rangle\rangle_{\mathcal{B}_A}
	\label{boundary adjoint}
	\end{align}
	for every $\omega\in \mathcal{B}_A^{1,\alpha}$ and $q\in\mathcal{A}_A^{1,\alpha}$.
	\subsection{Existence of solutions to the linearised problem}
	First-order elliptic boundary problems satisfy a version of the Fredholm alternative if the boundary operator satisfies a certain compatibility condition, the so-called Lopatinski-Shapiro condition; cf. \cite{wendlandelliptic}. A definition of this condition can be found in Chapter 4 of \cite{wendlandelliptic} for instance. It can be proved that both $\mathcal{L}$ and $\mathcal{L}^*$ satisfy this condition. We postpone the proof to the appendix.\\ \indent  We proceed to prove the following solvability criterion. From now on, all norms are computed with respect to the isothermal Euclidean coordinates on $D$.
	\begin{lem}
		\label{Fredholm}
		Let $q\in\mathcal{A}^{0,\alpha}_A$ and  $\psi \in C^{1,\alpha}(\partial D)$. Then there exists a one-form $w\in \mathcal{B}^{1,\alpha}_{A}$ with $\mathcal{L}(w)=(q,\psi)$ if and only if
		\begin{align}
		\langle q,\hat q \rangle_{\mathcal{A}_A}+\int_{\partial D} \psi R^*_1(\hat q)\,\text{d}{vol}_h=0
		\label{existence criterion2}
		\end{align}
		for every $\hat q\in \ker(\mathcal{L}^*)$. If (\ref{existence criterion2}) holds, then $w$ satisfies the estimate
		\begin{align}
		|w|_{C^{1,\alpha}(\bar{D})}\leq c( |q|_{C^{0,\alpha}(\bar{D})}+|\psi|_{C^{1,\alpha}(\partial D)}+|w|_{C^{0,\alpha}(\bar{D})}),
		\end{align}
		where $c$ depends on $\alpha$, $|A|_{C^{0,\alpha}(\bar{D})}$ and $|h|_{C^{1,\alpha}(\bar{D})}$. Moreover, if $q$ is of class $C^{b,\alpha}$ and $\psi$ of class $C^{b+1,\alpha}$ for some integer $1\leq b\leq k-1$, then $w$ is of class $C^{b+1,\alpha}$.
	\end{lem}
	\begin{proof}
		We can regard the operators $\mathcal{L}$ and $\mathcal{L}^*$ as mappings from $$C^{1,\alpha}(\bar{D},\mathbb{R}^2)\to C^{0,\alpha}(\bar{D},\mathbb{R}^2)\times C^{1,\alpha}(\partial D).$$ According to Lemma \ref{ellipticity} and Lemma \ref{lopatinski condition}.1, the operators $\mathcal{L}$ and $\mathcal{L}^*$ are elliptic operators that satisfy the Lopatinski-Shapiro condition. Consequently, they are  Fredholm operators, see \cite[Theorem 4.2.1]{wendlandelliptic}. The existence of a $C^{1,\alpha}$-solution under the given hypothesis now follows from the identity (\ref{boundary adjoint}) and \cite[Theorem 1.3.4]{wendlandelliptic}. Moreover, \cite[Theorem 4.1.2]{wendlandelliptic} provides us with the a-priori estimate. The claimed regularity follows from standard elliptic theory.		\end{proof}
	In order to complete the proof of the existence of solutions to the first four lines of the linearised equation (\ref{linearized equations}), we proceed to show that the kernel of $\mathcal{L}^*$ is empty. Let $\hat q\in\mathcal{A}_A^{1,\alpha}$ such that $\mathcal{L}^*(\hat q)=0$.  We first transform (\ref{adjoint operator}) into a more useful form.
	\begin{lem}
		The $(0,3)$ tensor $\nabla_h \hat q$ is symmetric. In particular, in any normal coordinate frame, there holds
		\begin{align}
		0=\partial_1\hat q_{22}-\partial_2\hat q_{12} \label{a1der}
		\end{align}
		as well as
		\begin{align}
		0=\partial_1\hat q_{12}-\partial_2\hat q_{11}. \label{a2der}
		\end{align}
	\end{lem}
	\begin{proof}
		See  \cite[p. 10]{li2016weyl}.
	\end{proof} 
	Let $\times:\mathbb{R}^3\to\mathbb{R}^3$ be the Euclidean cross product. Taking the cross product with the normal $\nu$ defines a linear map on the tangent bundle of $F(\bar{D})$. We define the $(1,1)$-tensor $Q$ via
	\begin{align*}
	Q=X_md^m=h^{ij}\hat q_{im} \nu\times \partial_jFd^m=\nu\times \iota_{dF}\hat q. 		\end{align*}
	Using the properties of the cross product, we find that
	\begin{align}
	X_1&=\frac{1}{\sqrt{\det(h)}}(-\hat q_{12}\partial_1F+\hat q_{11}\partial_2F), \label{A1}
	\\X_2&=\frac{1}{\sqrt{\det(h)}}(-\hat q_{22}\partial_1F+\hat q_{12}\partial_2F). \label{A2}
	\end{align}
	The next lemma is a variation of Lemma 7 in \cite{li2016weyl}. In its statement ant its proof, $\bar \nabla$ denotes the flat connection of $\mathbb{R}^3$. 
	\begin{lem}
		\label{maxprinciple}
		Let $Y\in C^1(\bar{D},\mathbb{R}^3)$ be a vector field satisfying $Y\cdot \partial_r F=0$ on $\partial D$ and 
		$$\omega=Q\cdot Y=X_i\cdot Y d^i.$$
		If $Y$ satisfies 
		\begin{eqnarray}\label{3.10}
		dF\cdot \bar{\nabla}Y=0
		\end{eqnarray}
		in the sense of symmetric $(0,2)$-tensors, then  $\omega= 0$. 
	\end{lem}
	\begin{proof}
		We compute
		\begin{eqnarray}
		d\omega&=&\partial_j(X_i\cdot Y)d^j\wedge d^i \nonumber\\
		&=&(\bar{\nabla}_{\partial_j}X_i\cdot Y+X_i\cdot \bar{\nabla}_{\partial_j}Y)d^j\wedge d^i\nonumber\\
		&=&((\bar{\nabla}_{\partial_1}X_2-\bar{\nabla}_{\partial_2}X_1)\cdot Y+X_2\cdot \bar{\nabla}_{\partial_1}Y-X_1\cdot \bar{\nabla}_{\partial_2}Y)d^1\wedge d^2\nonumber.
		\end{eqnarray}
		We fix a point $p\in D$ and choose normal coordinates centred at $p$.
		Equation \eqref{3.10} reads
		$$
		\partial_1F\cdot\bar{\nabla}_{\partial_1} Y=\partial_2F\cdot\bar{\nabla}_{\partial_2} Y=\partial_1F\cdot\bar{\nabla}_{\partial_2} Y+\partial_2F\cdot\bar{\nabla}_{\partial_1} Y=0.
		$$
		Together with $(\ref{A1})$ and $(\ref{A2})$, this implies that
		$$X_2\cdot \bar{\nabla}_{\partial_1}Y-X_1\cdot \bar{\nabla}_{\partial_2}Y=\hat q_{12}\partial_2 F\cdot \bar{\nabla}_{\partial_1}Y+\hat q_{12}\partial_1F\cdot \bar{\nabla }_{\partial_2}Y=0.$$ Moreover, we may choose the direction of the normal coordinates to be principal directions, that is, $A_{12}$=0. Using (\ref{A1}) and (\ref{A2}), we then find
		\begin{align*}
		\bar{\nabla}_{\partial_1}X_2-\bar{\nabla}_{\partial_2}X_1&=(-\partial_1\hat q_{22}+\partial_2\hat q_{12})\partial_1F+(\partial_1\hat q_{12}-\partial_2\hat q_{11})\partial_2F
		-(\hat q_{22}A_{11}+\hat q_{11}A_{22})\nu=0
		\end{align*}
		where we have used (\ref{a1der}), (\ref{a2der}) and $\operatorname{tr}_A(\hat q)=0$. It follows that  $\omega$ is closed. Since the disc is contractible, $\omega$ is also exact and consequently, there exists a function $\zeta\in C^{2,\alpha}(\bar{D})$ satisfying 
		$$
		\omega=d\zeta=\partial_l\zeta d^l.
		$$
		This implies that $\partial_l\zeta=X_l\cdot Y$. As $R^*_2(\hat q)=0$, we have $\hat q_{r\varphi}=0$ on $\partial D $ where $\varphi,r$ denote isothermal polar coordinates. It follows that $$\partial_\varphi\zeta=X_\varphi\cdot Y= \frac{1}{\sqrt{\det(h)}}\hat q_{\varphi\varphi}\partial_rF\cdot Y=0$$ on $\partial D$. In particular, $\zeta$ is constant on $\partial D$. We can now argue as in the proof of Lemma 7 in \cite{li2016weyl} to show that $\zeta$ satisfies a strongly elliptic equation with bounded coefficients. The  maximum principle implies that $\zeta$ attains its maximum and minimum on the boundary. Since $\zeta$ is constant on $\partial D$, $\zeta$ is constant on all of $\bar{D}$ and consequently $\omega=d\zeta=0$.
	\end{proof}
We now prove the following existence result.
	\begin{lem}
		Let $\tilde q\in C^{k-1,\alpha}(\bar{D},\operatorname{Sym}(T^*D\otimes T^*D))$ and $\psi\in C^{k,\alpha}(\partial D)$. Then there exists a map $\Psi\in C^{k,\alpha}(\bar{D},\mathbb{R}^3)$ which satisfies
		\begin{align*}
		d\Psi\cdot dF= \tilde q \text{ in } D, \\
		\Psi\cdot F =	\psi \text{ on } \partial D.
		\end{align*}	
		Moreover, the one-form $w=u_id^i$ where $u_i=\Psi\cdot \partial_iF$ satisfies the estimate
		\begin{align}
		|w|_{C^{1,\alpha}(\bar D)}\leq c( |\tilde q|_{C^{0,\alpha}(\bar D)}+|\psi|_{C^{1,\alpha}(\partial D)}+|w|_{C^0(\bar D)}),
		\label{aprioriest}
		\end{align}
		where $c$ depends on $|h|_{C^{2}(\bar{D})}$ as well as $|A|_{C^{1}(\bar{D})}$.
		\label{existlem}
	\end{lem}
	\begin{proof}
		We have seen that it is sufficient to prove the existence of a smooth one-form $w$ solving $\mathcal{L}(w)=(q,\psi)$, where $q=\tilde q-\frac12 \operatorname{tr}_A(\tilde q)A$. $w$ then uniquely determines the map $\Psi$. Moreover, there holds $|q|_{C^{0,\alpha}(\bar{D})}\leq c |\tilde q|_{C^{0,\alpha}(\bar{D})}$, where $c$ solely depends on $|A|_{C^1(\bar{D})}$. Hence, according to Lemma \ref{Fredholm} it suffices to show that $\ker{(\mathcal{L}^*)}$ is empty.\\ \indent Let $\hat q\in\ker(\mathcal{L}^*)$ and $Q$ be defined as above. Denote the standard basis of $\mathbb{R}^3$ by $\{e_1,e_2,e_3\}$ and define the vector fields $Y_a=e_a\times F$ for $a\in\{1,2,3\}$. Since $\partial_rF=EF$ on $\partial D$, there holds $Y_a\cdot \partial_rF=0$ on $\partial D$. We denote the group of orthogonal matrices by $O(3)$. It is well-known that $T_{\operatorname{Id}}(O(3))=\operatorname{Skew}(3)$ which is the space of skew-symmetric matrices. The operator $Y\mapsto e_a\times Y$ is skew-symmetric and it follows that there exists a $C^1$-family of orthogonal matrices $\mathcal{I}(t)$, where $t\in(-\epsilon,\epsilon)$ for some $\epsilon>0$, such that $\mathcal{I}(0)=\operatorname{Id}$ and $\mathcal{I}'(0)(F)=Y_a$. In particular, every $\mathcal{I}(t)(F)$ is a solution to the free boundary value problem (\ref{fbp}) for the metric $h$. It follows that
		\begin{align*}
		0=\frac12\frac{d}{dt} d\mathcal{I}(t)(F)\cdot d\mathcal{I}(t)(F)\bigg|_{t=0}=dY_a\cdot dF=\bar{\nabla}Y_a\cdot dF	
		\end{align*}  	
		as the connection $\bar{\nabla}$ and the vector-valued differential $d$ on $\mathbb{R}^3$ can be identified with each other. Hence the hypotheses of Lemma \ref{maxprinciple} are satisfied and $Y_a\cdot Q $ vanishes identically on $D$. Clearly, $Q\cdot \nu= 0$ and $Q=0$ implies $\hat q=0$. Thus, it suffices to show that $$\operatorname{span}\{\nu,Y_1,Y_2,Y_3\}=\mathbb{R}^3 $$on  $D$. \\ \indent  At any point $p\in D$, the vector fields $Y_1, Y_2, Y_3$ span the space $F(p)^\perp$ so we need to show that $\nu\notin F(p)^\perp$ on  $D$. Applying the Gauss-Bonnet formula we obtain
		\begin{align*}
		\operatorname{Length}(\partial D)+\int_D K\,\text{d}{vol}_h=2\pi,	
		\end{align*} 
		that is, $\operatorname{Length}(\partial D)< 2\pi$.  According to the Crofton formula, we have
		\begin{align*}
		\operatorname{Length}(F(\partial D))=\frac14\int_{\mathbb S^2} \#(F(\partial D)\cap \xi^\perp)\,d\xi.	
		\end{align*}
		The set of great circles which intersect $F(\partial D)$ exactly once is of measure zero. It follows that $F$ has to avoid a great circle and, after a rotation, we may assume that $F(\partial D)$ is contained in the upper hemisphere. Now, the strict convexity of $F$ implies that $F$ lies above the cone $$C=\{tF(p):t\in[0,1], p\in\partial D\}$$ and only touches $C$ on the boundary. On $\partial D$, there holds $\nu\cdot F=0$. If there was an interior point $p$ such that $\nu(p)\cdot F(p)=0$, then  the tangent plane $T_{F(p)}{F(\bar{D})}$ would meet a part of $\mathbb{S}^2$ above the cone $C$. However,  strict convexity then implies that $F$ has to lie on one side of that tangent plane. This contradicts $F(\partial D)=\partial C$. 
	\end{proof}	
	By induction, we are now able to find $C^2$-solutions $\Psi^l$ to the first four lines of (\ref{linearized equations}) which however do not yet satisfy the condition
	\begin{align*}
	\partial_r \Psi^l =\sum_{i=0}^{l}\binom{l}{i}\partial_t^iE_t\Psi^{l-i}\bigg|_{t=0} \text{ on } \partial D.	
	\end{align*} 
	We now show that this additional boundary condition is implied by the constancy of the geodesic curvature in the $t$-direction.
	\begin{lem}
		Let $\Psi^l$, $l\in\mathbb{N}$, be $C^2$-solutions of
		\begin{align}
		\begin{dcases}
		& dF\cdot d\Psi^l+\frac12 \sum_{i=1}^{l-1}\binom{l}{i}d\Psi^{l-i}\cdot d\Psi^{i}=\frac{1}{2}\operatorname{Id}\sum_{i=0}^l\binom{l}{i}\partial_t^iE_t\partial_t^{l-i}E_t\big|_{t=0} \text{ in } D,\\
		& F\cdot \Psi^l+\frac12 \sum_{i=1}^{l-1}\binom{l}{i}\Psi^{l-i}\cdot \Psi^{i}=0 \text{ on } \partial D.
		\label{equation improved bc}
		\end{dcases}	 	
		\end{align}
		Then there also holds 
		\begin{align*}
		\partial_r \Psi^l =\sum_{i=0}^{l}\binom{l}{i}\partial_t^iE_{t}\Psi^{l-i}\bigg|_{t=0} \text{ on } \partial D.	
		\end{align*}
		\label{improvedbc}	
	\end{lem}
	\begin{proof} We choose isothermal polar coordinates $(\varphi,r)$ centred at the origin and abbreviate $\partial_1=\partial_\varphi$ as well as $\partial_2=\partial_r$. Differentiating the boundary condition in (\ref{equation improved bc}) tangentially we find 
		\begin{align}
		\label{diffbc}
		\partial_1F\cdot \Psi^l+F\cdot \partial_1\Psi^l+\frac12 \sum_{i=1}^{l-1}\binom{l}{i}\big(\partial_1\Psi^{l-i}\cdot \Psi^{i}+\Psi^{l-i}\cdot \partial_1\Psi^{i})=0.
		\end{align}	
		At every point $p\in\partial D$, we need to show the following three equations
		\begin{align}
		\partial_2 \Psi^l\cdot F &=\bigg(\partial_t^lE_t+\sum_{i=0}^{l-2}\binom{l}{i}\partial_t^iE_{t}\Psi^{l-i}\cdot F \bigg)\bigg|_{t=0} \label{bc1}, \\
		\partial_2 \Psi^l\cdot \partial_1F &=\sum_{i=0}^{l-1}\binom{l}{i}\partial_t^iE_t\Psi^{l-i}\cdot \partial_1F\bigg|_{t=0} \label{bc2}, \\
		\partial_2 \Psi^l\cdot \nu &=\sum_{i=0}^{l-1}\binom{l}{i}\partial_t^iE_t\Psi^{l-i}\cdot \nu \bigg|_{t=0} \label{bc3},
		\end{align}
		where we have used that $|F|^2=1$, $F\cdot\Psi^1=0$ and $F\cdot \partial_1F=F\cdot \nu=0$ on $\partial D$. We prove the statement by induction and start with $l=1$. Since $\partial_2 F=EF$ we find
		\begin{align*}
		\partial_2\Psi^1\cdot F=\frac1E\partial_2\Psi^1\cdot \partial_2F=\partial_tE_{t}\big|_{t=0}
		\end{align*}
		by the differential equation (\ref{equation improved bc}). This proves (\ref{bc1}). Next, we have
		\begin{align*}
		\partial_2\Psi^1\cdot \partial_1F=-\partial_1\Psi^1\cdot \partial_2F=-E\partial_1\Psi^1\cdot F=E\Psi^1\cdot \partial_1F
		\end{align*}
		where we have used the equation (\ref{equation improved bc}), the free boundary condition of $F$ and (\ref{diffbc}). This proves (\ref{bc2}).
		Now let $\zeta=rE_t$. Since $k_h(\partial D)= 1$ for all $t\in[0,1]$, there holds $1=-\partial_2(\zeta^{-1})$ on $\partial D$.  Differentiating in time, we obtain 
		\begin{align*}
		0=-\partial_t\partial_2\bigg(\frac{1}{\zeta}\bigg)=\partial_2\bigg(\frac{\partial_t\zeta}{\zeta^2}\bigg)=\bigg(\frac{\partial_t\partial_2\zeta}{\zeta^2}-2\frac{\partial_t\zeta\partial_2\zeta}{\zeta^3}\bigg).
		\end{align*}
		Since $\partial_2 \zeta=\zeta^2$ at $t=0$, there holds $$\zeta\partial_t\partial_2\zeta=2\zeta^2\partial_t\zeta.$$ According to (\ref{equation improved bc}), we have $\zeta\partial_t\zeta=\partial_1F\cdot\partial_1\Psi^1$ at $t=0$. Differentiating yields 
		$$\zeta\partial_t\partial_2\zeta+\zeta^2\partial_t \zeta=\partial_1\partial_2F\cdot \partial_1\Psi^1+\partial_1F\cdot \partial_1\partial_2\Psi^1$$ at $t=0$. Using that $\zeta=E_t$ on $\partial D$, $E_t=E$ for $t=0$ and $$\partial_1\partial_2F\cdot \partial_1\Psi^1=\partial_1(EF)\cdot\partial_1\Psi^1=\partial_1EF\cdot \partial_1\Psi^1+E\partial_1F\cdot\partial_1\Psi=\partial_1EF\cdot \partial_1\Psi^1+E^2\partial_tE_{t}\big|_{t=0},$$ we obtain 
		\begin{align}
		\partial_1F\cdot \partial_1\partial_2\Psi^1=2\partial_tE_{t}E^2\big|_{t=0}-\partial_1EF\cdot 
		\partial_1\Psi^1. \label{3bc ident 1}
		\end{align} 
		We compute
		\begin{align*}
		\partial_1\partial_2\Psi^1\cdot \partial_1F&=\partial_1(\partial_2\Psi^1\cdot \partial_1F)-\partial_2\Psi^1\cdot \partial_1\partial_1F \\	
		&=-\partial_1(\partial_1\Psi^1\cdot \partial_2F)-\partial_2\Psi^1\cdot \partial_1\partial_1F\\
		&=-\partial_1(\partial_1\Psi^1\cdot EF)-\partial_2\Psi^1\cdot \partial_1\partial_1F\\
		&=\partial_1(\Psi^1\cdot E\partial_1F)-\partial_2\Psi^1\cdot \partial_1\partial_1F\\
		&=\partial_1E \Psi^1\cdot \partial_1F +E\Psi^1 \cdot \partial_1\partial_1F +E \partial_1F\cdot \partial_1\Psi^1-\partial_2\Psi^1\cdot \partial_1\partial_1F\\
		&=-\partial_1E \partial_1\Psi^1\cdot F +E\Psi^1 \cdot \partial_1\partial_1F +E^2\partial_tE_{t}\big|_{t=0}-\partial_2\Psi^1\cdot \partial_1\partial_1F
		\end{align*}
		which together with (\ref{3bc ident 1}) implies that
		\begin{align}
		\partial_1\partial_1F\cdot ({-\partial_2\Psi^1+E\Psi^1})-\partial_tE_{t}E^2\big|_{t=0}=0.	
		\end{align}
		A calculation shows that $\partial_1\partial_1F=-A_{11}\nu+\partial_1EE^{-1}\partial_1F-E\partial_2F$. Furthermore, we have already shown that 
		$$
		-\partial_2\Psi^1+E\Psi^1=-\partial_tE_{t}\big|_{t=0}F+\nu\cdot(-\partial_2\Psi^1+E\Psi^1+\partial_tE_{t}F\big|_{t=0})\nu =-\partial_tE_{t}\big|_{t=0}F+\nu\cdot(-\partial_2\Psi^1+E\Psi^1) \nu.
		$$
		Together with the above, $F\cdot \partial_2F=E$ and $F\cdot\partial_1F=0$, this implies that
		$$
		A_{11}(\partial_2\Psi^1-E\Psi^1)\cdot \nu=0.
		$$
		This proves the claim since $A_{11}>0$ and $E=E_0$.\\ \indent  Now, let $l>1$ and suppose that the assertion has already been shown for every $i<l$.  At $t=0$, there holds
		\begin{align*}
		\partial_2\Psi^l\cdot F\notag =&\frac1E\partial_2\Psi^l\cdot \partial_2F\\
		=&\frac{1}{2E}\bigg(\sum_{i=0}^l \binom{l}{i}\partial_t^iE_{t}\partial_t^{l-i}E_{t}-\sum_{i=1}^{l-1}\binom{l}{i}\partial_2\Psi^{l-i}\cdot\partial_2\Psi^i\bigg)\\	
		=&\frac{1}{2E}\bigg(\sum_{i=0}^l \binom{l}{i}\partial_t^iE_{t}\partial_t^{l-i}E_{t}-\sum_{i=1}^{l-1}\binom{l}{i}\bigg[\sum_{j=0}^{l-i}\binom{l-i}{j}\partial_t^{j}E_{t}\Psi^{l-i-j}\bigg]\cdot\bigg[\sum_{m=0}^{i}\binom{i}{m}\partial_t^{m}E_{t}\Psi^{i-m}\bigg]\bigg)\\
		=&\frac{1}{2E}\bigg(2E\partial_t^{l}E_{t}-2\sum_{i=1}^{l-1}\binom{l}{i}\partial_t^{l-i}E_{t}\sum_{j=0}^{i-1}\binom{i}{j}\partial_t^{j}E_{t}F\cdot\Psi^{i-j}
		\\&-\sum_{i=1}^{l-1}\binom{l}{i} \sum_{j=0}^{l-i-1}\sum_{m=0}^{i-1}\binom{l-i}{j}\binom{i}{m}\partial_t^{j}E_{t}\partial_t^{m}E_{t}\Psi^{l-i-j}\cdot\Psi^{i-m}\bigg)
		\\=&\frac{1}{2E}\bigg(2E\partial_t^{l}E_{t}-2\sum_{i=1}^{l-1}\sum_{j=0}^{l-i-1}\binom{l}{i+j}\binom{i+j}{j}\partial_t^{i}E_{t}\partial_t^{j}E_{t}F\cdot\Psi^{l-i-j}
		\\& - \sum_{i=0}^{l-2}\sum_{j=0}^{l-2-i}\binom{l}{i+j}\binom{i+j}{j}\partial_t^{i}E_{t}\partial_t^{j}E_{t}\sum_{m=1}^{l-1-i-j}\binom{l-i-j}{m}\Psi^{l-i-j-m}\cdot\Psi^m \bigg)
		\\=&\frac{1}{2E}\bigg(2E\partial_t^{l}E_{t}-2\sum_{i=1}^{l-1}\binom{l}{i+l-i-1}\binom{i+l-i-1}{l-i-1}\partial_t^{i}E_{t}\partial_t^{l-i-1}E_{t}F\cdot\Psi^{l-i-(l-i-1)}
		%\\&-\sum_{j=0}^{l-1}\binom{l}{l+j}\binom{l+j}{j}E_{0,l}E_{0,j}F\cdot\Psi^{l-(l-1)-j}
		\\& + 2\sum_{j=0}^{l-2}\binom{l}{j}\binom{j}{0}E_{t}\partial_t^{j}E_{t}F\cdot\Psi^{l-j} \bigg)
		\\=& \partial_t^{l}E_{t}+\sum_{j=0}^{l-2}\binom{l}{j}\partial_t^{j}E_{t}F\cdot\Psi^{l-j},
		\end{align*}	
		which proves (\ref{bc1}).  In the first equation, we used that $F$ satisfies the free boundary condition. In the second equation, we used the $(2,2)$-component of the differential equation for $\Psi^l$. In the third equation, we used that the claim holds for every $i<l$. The fourth equation follows from extracting the terms involving $F$ and using $|F|^2=1$ and $\Psi^0=F$. The fifth equation follows from rearranging the sums, changing indices and rewriting the binomial coefficients. In the sixth equation, we used the boundary condition for each $\Psi^{l-i-j-m}$ and cancel the terms which appear twice. Finally, we used that $\Psi^1\cdot F=0$. We proceed to show (\ref{bc2}). We have
		\begin{align*}
		\partial_2\Psi^l\cdot \partial_1F=&-\partial_1\Psi^l\cdot \partial_2F -\frac12\sum_{i=1}^{l-1}\binom{l}{i}\bigg(\partial_1\Psi^{l-i}\cdot\partial_2\Psi^i+\partial_2\Psi^{l-i}\cdot\partial_1\Psi^i\bigg) \\
		=&-E\partial_1\Psi^l\cdot F-\frac12 \sum_{i=1}^{l-1}\binom{l}{i}\bigg(\partial_1\Psi^{l-i}\cdot\partial_2\Psi^i+\partial_2\Psi^{l-i}\cdot\partial_1\Psi^i\bigg) \\
		=&E\Psi^l\cdot \partial_1F+\frac E2\sum_{i=1}^{l-1}\binom{l}{i}\bigg(\partial_1\Psi^{l-i}\cdot\Psi^{i}+\Psi^{l-i}\cdot\partial_1\Psi^{i}\bigg) \\
		&-\frac12\sum_{i=1}^{l-1}\binom{l}{i}\bigg(\sum_{j=0}^{i}\binom{i}{j}\partial_t^{j}E_{t}\partial_1\Psi^{l-i}\cdot\Psi^{i-j}+\sum_{j=0}^{l-i}\binom{l-i}{j}\partial_t^{j}E_{t}\partial_1\Psi^{i}\cdot\Psi^{l-i-j}\bigg) \\
		=&E\Psi^l\cdot \partial_1F-\frac12\sum_{i=1}^{l-1}\binom{l}{i}\partial_t^{i}E_{t}\bigg(\sum_{j=1}^{i}\binom{i}{j}\partial_1\Psi^{l-i}\cdot\Psi^{i-j}+\sum_{j=1}^{l-i}\binom{l-i}{j}\partial_1\Psi^{i}\cdot\Psi^{l-i-j}\bigg)
		\\
		=& E\Psi^l\cdot \partial_1F -\frac12 \sum_{i=1}^{l-1}\binom{l}{i}\partial_t^{i}E_{t}\bigg(\sum_{j=0}^{l-i-1}\binom{l-i}{j}\partial_1\Psi^{l-i-j}\cdot\Psi^j+\sum_{j=1}^{l-i}\binom{l-i}{j}\partial_1\Psi^{l-i-j}\cdot\Psi^j\bigg) \\
		=&E\Psi^l\cdot \partial_1F + \sum_{i=1}^{l-1}\binom{l}{i}\partial_t^{i}E_{t}\partial_1F\cdot\Psi^{l-i},
		\end{align*}
		which proves (\ref{bc2}). In the first equation, we  used the $(1,2)$ component of the differential equation for $\Psi^l$. In the second equation, we used the free boundary condition. The third equation follows from the differentiated boundary condition (\ref{diffbc}) for $\Psi^l$ and the induction hypothesis. The fourth equation, follows from using the $j=0$ terms to cancel  the second term. The fifth equation follows from changing indices and recomputing the binomial coefficients. The last equation follows from the differentiated boundary condition for $\Psi^{l-1}$.  \\ \indent We recall that $\zeta=rE_t$. Differentiating the identity
		$
		\partial_2\zeta=\zeta^2
		$
		$l$ times with respect to $t$,  we obtain
		$$
		\partial_2\partial_t^l\zeta\bigg|_{t=0}=\partial_t^l\zeta^2\bigg|_{t=0}=\sum_{i=0}^l\binom{l}{i}\partial_t^i\zeta\partial_t^{l-i}\zeta\bigg|_{t=0}.
		$$
Moreover,
		$$
		\partial_t^l\zeta \zeta\bigg|_{t=0}  =\frac12\partial_t^l \zeta^2\bigg|_{t=0}-\frac12\sum_{i=1}^{l-1}\binom{l}{i}\partial_t^i \zeta \partial_t^{l-i}\zeta\bigg|_{t=0}.
		$$
		Hence, using $\zeta=E$ on $\partial D$, the identity $\partial_2\zeta=\zeta^2$ for the lower order terms and the $(1,1)$ component of the differential equation for $\Psi^l$, we obtain, at $t=0$, that
		\begin{align*}
		&\partial_2\bigg(\partial_1F\cdot\partial_1\Psi^l+\frac12 \sum_{i=1}^{l-1}\binom{l}{i}\partial_1\Psi^{l-i}\cdot\partial_1\Psi^i\bigg)\\=&\zeta\sum_{i=0}^{l}\binom{l}{i}\partial_t^i\zeta\partial_t^{l-i}\zeta +\frac12\sum_{i=1}^{l-1}\binom{l}{i}\bigg(\partial_t^{l-i}\zeta\sum_{j=0}^{i}\binom{i}{j}\partial_t^j\zeta\partial_t^{i-j}\zeta+\partial_t^{i}\zeta\sum_{m=0}^{l-i}\binom{l-i}{m}\partial_t^m\zeta\partial_t^{l-i-m}\zeta\bigg).
		\end{align*}
		Consequently, at $t=0$, we have 
		\begin{align}
		&\partial_1\partial_2F\cdot\partial_1\Psi^l+\partial_1F\cdot\partial_1\partial_2\Psi^l+\frac12 \sum_{i=1}^{l-1}\binom{l}{i}\bigg(\partial_1\partial_2\Psi^{l-i}\cdot\partial_1\Psi^i+\partial_1\Psi^{l-i}\cdot\partial_1\partial_2\Psi^i\bigg)\notag\\=&\zeta\sum_{i=0}^{l}\binom{l}{i}\partial_t^i\zeta\partial_t^{l-i}\zeta +\frac12\sum_{i=1}^{l-1}\binom{l}{i}\bigg(\partial_t^{l-i}\zeta\sum_{j=0}^{i}\binom{i}{j}\partial_t^j\zeta\partial_t^{i-j}\zeta+\partial^i_t\zeta\sum_{m=0}^{l-i}\binom{l-i}{m}\partial_t^m\zeta\partial_t^{l-i-m}\zeta\bigg).
		\label{firsteqncomp}
		\end{align}
		We calculate 
		\begin{align}
		\partial_1\partial_2F\cdot \partial_1\Psi^l=\partial_1EF\cdot\partial_1\Psi^l+E\partial_1F\cdot\partial_1\Psi^l	\label{f12terms}
		\end{align}
		as well as
		\begin{align}
		&\partial_1F\cdot\partial_1\partial_2\Psi^l\notag\\=&\partial_1(\partial_1F\cdot\partial_2\Psi^l)-\partial_1\partial_1F\cdot\partial_2\Psi^l \notag
		\\=&-\partial_1(\partial_1\Psi^l\cdot \partial_2F) -\partial_1\bigg(\frac12 \sum_{i=1}^{l-1}\binom{l}{i}(\partial_1\Psi^{l-i}\cdot\partial_2\Psi^i+\partial_2\Psi^{l-i}\cdot\partial_1\Psi^i)\bigg) -\partial_1\partial_1F\cdot\partial_2\Psi^l \notag
		\\=& -\partial_1(E\partial_1\Psi^l\cdot F)-\partial_1\bigg(\frac12 \sum_{i=1}^{l-1}\binom{l}{i}(\partial_1\Psi^{l-i}\cdot\partial_2\Psi^i+\partial_2\Psi^{l-i}\cdot\partial_1\Psi^i)\bigg)-\partial_1\partial_1F\cdot\partial_2\Psi^l \notag
		\\=&\partial_1\bigg(E\partial_1F\cdot \Psi^l +\frac E2\sum_{i=1}^{l-1}\binom{l}{i}(\partial_1\Psi^{l-i}\cdot\Psi^i+\Psi^{l-i}\cdot\partial_1\Psi^i)\bigg)
		\\&-\partial_1\bigg(\frac12 \sum_{i=1}^{l-1}\binom{l}{i}(\partial_1\Psi^{l-i}\cdot\partial_2\Psi^i+\partial_2\Psi^{l-i}\cdot\partial_1\Psi^i)\bigg)-\partial_1\partial_1F\cdot\partial_2\Psi^l_2 \notag
		\\=&\partial_1\partial_1F\cdot \notag (E\Psi^l-\partial_2\Psi^l)+E\partial_1F\cdot\partial_1\Psi^l\\&+\partial_1E\bigg(\partial_1F\cdot \Psi^l+\frac12\sum_{i=1}^{l-1}\binom{l}{i}(\partial_1\Psi^{l-i}\cdot\Psi^i+\Psi^{l-i}\cdot\partial_1\Psi^i)\bigg) \notag 
		\\&+\frac12\sum_{i=1}^{l-1}\binom{l}{i}\bigg[E\bigg(\partial_1\partial_1\Psi^{l-i}\cdot\Psi^i+\Psi^{l-i}\cdot\partial_1\partial_1\Psi^i)+E\bigg(\partial_1\Psi^{l-i}\cdot\partial_1\Psi^i+\partial_1\Psi^{l-i}\cdot\partial_1\Psi^i\bigg) \notag
		\\&-\bigg(\partial_1\partial_1\Psi^{l-i}\cdot\partial_2\Psi^i+\partial_2\Psi^{l-i}\cdot\partial_1\partial_1\Psi^i\bigg)-\bigg(\partial_1\Psi^{l-i}\cdot\partial_1\partial_2\Psi^i+\partial_1\partial_2\Psi^{l-i}\cdot\partial_1\Psi^i\bigg)\bigg] \notag
		\\=&\partial_1\partial_1F\cdot (E\Psi^l-\partial_2\Psi^l)+E\partial_1F\cdot\partial_1\Psi^l+I+II+III+IV+V. \label{tau12terms}
		\end{align}
		In the second equation, we used the $(1,2)$ component of the differential equation for $\Psi^l$ and in the fourth equation, we used the differentiated boundary condition. Now $I$ and the first term on the left hand side of (\ref{f12terms}) cancel out because of the differentiated boundary condition, $V$ cancels out the third term on the left hand side of (\ref{firsteqncomp}). Moreover, at $t=0$, we have 
		$$
		2E\partial_1F\cdot\partial_1\Psi^l+III=	E\sum_{i=0}^{l}\binom{l}{i}\partial_t^{i}E_{t}\partial_t^{l-i}E_{t}.
		$$
		Combining all of this and noting that $E=\zeta$ on $\partial D$, we are left with, again at $t=0$,
		\begin{align}
		&\notag\partial_1\partial_1F\cdot (E\Psi^l-\partial_2\Psi^l) +II+IV\\=&\frac12\sum_{i=1}^{l-1}\binom{l}{i}\bigg(\partial_t^{l-i}E_{t}\sum_{j=0}^{i}\binom{i}{j}\partial_t^{j}E_{t}\partial_t^{i-j}E_{t}+\partial_t^{i}E_{t}\sum_{m=0}^{l-i}\binom{l-i}{m}\partial_t^{m}E_{t}\partial_t^{l-i-m}E_{t}\bigg).	\label{eqwrhs}
		\end{align}
		We now proceed to calculate the term $II+IV$. Using the induction hypothesis and changing the order of summation, we deduce at $t=0$
		\begin{align*}
		II+IV=&-\frac12 \sum_{i=1}^{l-1}\binom{l}{i}\bigg(\sum_{j=1}^{i}\binom{i}{j}\partial_t^{j}E_{t}\partial_1\partial_1\Psi^{l-i}\cdot\Psi^{i-j} 
		+\sum_{j=1}^{l-i}\binom{l-i}{j}
		\partial_t^{j}E_{t}\Psi^{l-i-j}\cdot\partial_1\partial_1\Psi^i\bigg)
		\\=&-\frac12\sum_{j=1}^{l-1}\binom{l}{j}\partial_t^{j}E_{t}\bigg(\sum_{i=0}^{l-j-1}\binom{l-j}{i}\partial_1\partial_1\Psi^{l-j-i}\cdot\Psi^{i}
		+\sum_{i=1}^{l-j}\binom{l-j}{i}\Psi^{l-j-i}\cdot\partial_1\partial_1\Psi^{i}\bigg)	
		\\=&-\frac12\sum_{j=1}^{l-1}\binom{l}{j}\partial_t^{j}E_{t}\bigg(\sum_{i=0}^{l-j}\binom{l-j}{i}\partial_1\partial_1\Psi^{l-j-i}\cdot\Psi^{i}
		+\sum_{i=0}^{l-j}\binom{l-j}{i}\Psi^{l-j-i}\cdot\partial_1\partial_1\Psi^{i}\bigg)	
		\\&+\sum_{j=1}^{l-1}\binom{l}{j}\partial_t^{j}E_{t}\partial_1\partial_1F\cdot \Psi^{l-j}.
		\end{align*}
		Differentiating the boundary condition (\ref{diffbc}) again and using the $(1,1)$ component of the differential equation for $\Psi^{l-j}$ we conclude
		\begin{align*}
		&-\frac12\sum_{j=1}^{l-1}\binom{l}{j}\partial_t^{j}E_{t}\bigg(\sum_{i=0}^{l-j}\binom{l-j}{i}\partial_1\partial_1\Psi^{l-j-i}\cdot\Psi^{i}
		+\sum_{i=0}^{l-j}\binom{l-j}{i}\Psi^{l-j-i}\cdot\partial_1\partial_1\Psi^{i}\bigg)
		\\=&\frac12\sum_{j=1}^{l-1}\binom{l}{j}\partial_t^{j}E_{t}\bigg(\sum_{i=0}^{l-j}\binom{l-j}{i}\partial_1\Psi^{l-j-i}\cdot\partial_1\Psi^{i}
		+\sum_{i=0}^{l-j}\binom{l-j}{i}\partial_1\Psi^{l-j-i}\cdot\partial_1\Psi^{i}\bigg)
		\\=&\frac12\sum_{j=1}^{l-1}\binom{l}{j}\partial_t^{j}E_{t}\bigg(\sum_{i=0}^{l-j}\binom{l-j}{i}\partial_t^{l-j-i}E_{t}\partial_t^{i}E_{t}
		+\sum_{i=0}^{l-j}\binom{l-j}{i}\partial_t^{l-j-i}E_{t}\partial_t^{i}E_{t}\bigg),
		\end{align*}
		which is exactly the right hand side of (\ref{eqwrhs}).  Hence, we conclude, at $t=0$, that
		$$
		0= \partial_1\partial_1F\cdot\bigg(E\Psi^l-\partial_2\Psi^l+\sum_{i=1}^{l-1}\binom{l}{i}\partial_t^{i}E_{t} \Psi^{l-i}\bigg)
		= \partial_1\partial_1F\cdot\bigg(-\partial_2\Psi^l+\sum_{i=0}^{l-1}\binom{l}{i}\partial_t^{i}E_{t} \Psi^{l-i}\bigg).
		$$
		Since we already know that the only potentially non-zero component of the second term in the product is the normal component, we conclude as before that
		$$
		0=A_{11}\nu\cdot\bigg(\partial_2\Psi^l-\sum_{i=0}^{l-1}\binom{l}{i}\partial_t^{i}E_{t}\cdot \Psi^{l-i}\bigg).
		$$
	\end{proof}
	\subsection{A-priori estimates and convergence of the power series}
	The final ingredient  to prove that the solution space is open is an a-priori estimate for the map $\Psi^l$. This will establish the convergence of the power series.
	\begin{lem}
		Let $l\in\mathbb{N}$ and $\Psi^l$ be a solution of (\ref{linearizedequations}). Then  $\Psi^l$ satisfies the estimate
		\begin{align*}
		|\Psi^l|_{C^{2,\alpha}(\bar{D})}\leq c\bigg(&
		\sum_{i=0}^l\binom{l}{i}|\partial_t^iE_{t}|_{C^{2,\alpha}(\bar{D})}|\partial_t^{l-i}E_{t}|_{C^{2,\alpha}(\bar{D})}+ \sum_{i=1}^{l-1}\binom{l}{i}|\Psi^{l-i}|_{C^{2,\alpha}(\bar{D})}\cdot |\Psi^{i}|_{C^{2,\alpha}(\bar{D})}
		\\&+\sum_{i=1}^{l}\binom{l}{i}|\partial_t^iE_{t}|_{C^{1,\alpha}(\bar{D})}|\Psi^{l-i}|_{C^{1,\alpha}(\bar{D})}
		\bigg)\bigg|_{t=0},
		\end{align*}
		where $c$  depends on $\alpha$, the $C^{3,\alpha}$-data of $F$ and the $C^{2,\alpha}$-data of $h$. As usual, all norms are taken with respect to Euclidean isothermal coordinates on $D$.
		\label{aprioriestimate}
	\end{lem}
	\begin{proof}
		We fix an integer $l\in\mathbb{N}$, abbreviate $\Psi=\Psi^l$ and define 
		$$
		\psi=-\frac12 \sum_{i=1}^{l-1}\binom{l}{i}\Psi^{l-i}\cdot \Psi^{i}, \qquad \Phi=\sum_{i=1}^{l}\binom{l}{i}\partial_t^iE_{t}\Psi^{l-i}\bigg|_{t=0}
		$$
		as well as
		$$
		q=\frac12\bigg(\operatorname{Id}\sum_{i=0}^l\binom{l}{i}\partial_t^iE_{t}\partial_t^{l-i}E_{t}- \sum_{i=1}^{l-1}\binom{l}{i}d\Psi^{l-i}\cdot d\Psi^{i}\bigg)\bigg|_{t=0}.
		$$
		Let $\partial_1,\partial_2$ be any coordinate system on $D$. We recall that $$w=u_id^i,\quad  u_i=\Psi\cdot \partial_i F, \quad v=\Psi\cdot \nu.$$ $\psi$ and $\Phi$ are of class $C^2$ while $q$ is of class $C^1$. Thus,  Lemma (\ref{existlem}) implies the a-priori estimate
		$$
		|w|_{C^{1,\alpha}(\bar{D})}\leq c( | q|_{C^{0,\alpha}(\bar{D})}+|\psi|_{C^{1,\alpha}(\partial D)}+|w|_{C^{0,\alpha}(\bar{D})}),
		$$
		where $c$ depends on $\alpha$, $|A|_{C^1(\bar{D})}$ and $|h|_{C^1(\bar{D})}$. After
		choosing $w$ to be orthogonal to the kernel of $\mathcal{L}$, a standard compactness argument implies the improved estimate
		\begin{align}
		|w|_{C^{1,\alpha}(\bar{D})}\leq c( | q|_{C^{0,\alpha}(\bar{D})}+|\psi|_{C^{1,\alpha}(\partial D)})
		\label{improvedest}.
		\end{align}
		We rewrite (\ref{linearized equations}) as
		\begin{align}
		\frac{1}{A_{11}}(\partial_1u_{1}-\Gamma^i_{1,1}u_i-q_{11})&=\frac{1}{A_{22}}(\partial_2u_{2}-\Gamma^i_{2,2}u_i-q_{22}), \notag	\\
		\partial_1u_{2}+\partial_2u_{1}-2\Gamma^i_{1,2}u_i-2q_{12}&=2\frac{A_{12}}{A_{11}}(\partial_1u_{1}-\Gamma^i_{1,1}u_i-q_{11}), \label{neumannbc}
		\\\partial_1u_{2}+\partial_2u_{1}-2\Gamma^i_{1,2}u_i-2q_{12}&=2\frac{A_{12}}{A_{22}}(\partial_2u_{2}-\Gamma^i_{2,2}u_i-q_{22}) \notag
		\end{align}
		and
		\begin{align}
		\begin{cases}
		& w(\partial_r)=E\psi \quad\hspace{0.1cm}\text{ on } \partial D \\
		& \partial_r \Psi =E\Psi+\Phi \text{ on } \partial D.\end{cases}	
		\label{pde} 
		\end{align}
		For the sake of readability, we define $Q$ to be a quantity which can be written in the form
		$$Q(\zeta)=\rho^{ijm}\zeta_{ijm},$$ where $\rho^{ijm}$ and $\zeta_{ijm}$ are functions defined on $D$. In practice, $\zeta$ will be one of $u_i$, $q_{ij},$ $\partial_i u_j$, $\partial_i q_{jm}$ and for every $\rho$ under consideration, we have the uniform estimate
		$$
		|Q(\zeta)|_{C^{0,\alpha}(\bar{D})}\leq c|\zeta|_{C^{0,\alpha}(\bar{D})},
		$$
		where $c$ only depends on $|h|_{C^{2,\alpha}(\bar{D})}$ and $|A|_{C^{1,\alpha}(\bar{D})}$. Differentiating the first equation of (\ref{neumannbc}) with respect to $\partial_1$ and the second equation with respect to $\partial_2 $, we obtain
		\begin{align*}
		\partial_1\partial_1u_1&=\frac{A_{11}}{A_{22}}\partial_1\partial_2u_2+Q(u_i)+Q(\partial_i u_j)+Q(q_{ij})+Q(\partial_iq_{jm}),\\
		\partial_2\partial_2u_1&=-\partial_2\partial_1u_2+2\frac{A_{12}}{A_{11}}\partial_2\partial_1u_1
		+Q(u_i)+Q(\partial_i u_j)+Q(q_{ij})+Q(\partial_iq_{jm}).
		\end{align*}
		Multiplying the second inequality by $A_{11}/A_{22}$ and adding the equalities we infer
		\begin{align}
		\partial_1\partial_1u_1+\frac{A_{11}}{A_{22}}\partial_2\partial_2u_1-2\frac{A_{12}}{A_{22}}\partial_1\partial_2u_1=Q(u_i)+Q(\partial_i u_j)+Q(q_{ij})+Q(\partial_iq_{jm}).
		\label{u1equation}
		\end{align}
		Similarly, one obtains 
		\begin{align}
		\partial_1\partial_1u_2+\frac{A_{11}}{A_{22}}\partial_2\partial_2u_2-2\frac{A_{12}}{A_{22}}\partial_1\partial_2u_2=Q(u_i)+Q(\partial_i u_j)+Q(q_{ij})+Q(\partial_iq_{jm}).\label{u2equation}
		\end{align}
		Since strict convexity implies that these equations are strongly elliptic, we may choose Euclidean coordinates and appeal to the Schauder theory for elliptic equations to obtain the interior estimate
		\begin{align}
		|w|_{C^{2,\alpha}(D_{1/2})}
		\leq c(|w|_{C^{1,\alpha}(\bar{D})}+|q|_{C^{1,\alpha}(\bar{D})}),
		\label{interiorest}
		\end{align}
		where $c$ depends on $\alpha$, $|h|_{C^{2,\alpha}(\bar{D})}$, $|A|_{C^{1,\alpha}(\bar{D})}$ and $|A^{-1}|_{C^{0}(\bar{D})}$.  Next, we choose polar coordinates $(\varphi,r)$ and note that $$u_\varphi=(u_{x_1},u_{x_2})\cdot(-x_2,x_1), \quad u_r=r^{-1}(u_{x_1},u_{x_2})\cdot(x_1,x_2)$$  are well-defined functions on the annulus $D\setminus D_{1/4}$. Here,  $(x_1,x_2)$ denote  Euclidean coordinates on $D$. The ellipticity of (\ref{u1equation}) and (\ref{u2equation}) remains unchanged if we write $\partial_\varphi,$ $\partial_r$ in terms of $\partial_{x_1},$ $\partial_{x_2}$. Consequently, $u_r$ and $u_\varphi$ satisfy a strongly elliptic equation with respect to Euclidean coordinates. Moreover, for every $b\in\mathbb{N}$, there is a constant $c_b>1$ which only depends on $b$ such that
		\begin{align}
		\label{equivalent norms}
		|w|_{c^{b,\alpha}(D\setminus D_{1/4})}\leq c_b(|u_r|_{C^{b,\alpha}(D\setminus D_{1/4})}+|u_\varphi|_{C^{b,\alpha}(D\setminus D_{1/4})})\leq c_b^2|w|_{C^{b,\alpha}(D\setminus D_{1/4})}.
		\end{align}
		We now apply the Schauder theory for elliptic equations to $u_r$ and $u_\varphi$. Since $u_r=\psi$ on $\partial D$, we find 
		\begin{align}
		|u_r|_{C^{2,\alpha}(D\setminus D_{1/4})}\leq c (|w|_{C^{1,\alpha}(D\setminus D_{1/4})}+|q|_{C^{1,\alpha}(D\setminus D_{1/4})}+|\psi|_{C^{2,\alpha}(\bar{D})}+|w|_{C^{2,\alpha}(\bar{D}_{1/4})}),
		\label{urest}
		\end{align} 
		where $c$ has the same dependencies as before. 
		Conversely, equation (\ref{neumannbc}) yields  that
		$$
		\partial_ru_\varphi=Q(u_i)+Q(\partial_i u_r)+Q(q_{ij}).\
		$$
		on $\partial D$.
		Consequently, the Schauder estimates for equations with Neumann boundary conditions give
		$$
		|u_\varphi|_{C^{2,\alpha}(D\setminus D_{1/4})}\leq c (|w|_{C^{1,\alpha}(D\setminus D_{1/4})}+|q|_{C^{1,\alpha}(D\setminus D_{1/4})}+|u_r|_{C^{2,\alpha}(D\setminus D_{1/4})}+|w|_{C^{2,\alpha}(D_{1/4})}).
		$$
		Combining this with (\ref{urest}), we may remove the $u_r$ term on the right hand side. Then, the interior estimate (\ref{interiorest}), the estimate (\ref{improvedest}) and (\ref{equivalent norms}) imply
		\begin{equation}
		\begin{aligned}
		|w|_{C^{2,\alpha}(\bar{D})}&\leq c (|w|_{C^{2,\alpha}({ D\setminus D_{1/2})}}+|u_\varphi|_{C^{2,\alpha}(D\setminus D_{1/4})}+|u_r|_{C^{2,\alpha}(D\setminus D_{1/4})})\\&\leq c(|q|_{C^{1,\alpha}(\bar{D})}+|\psi|_{C^{2,\alpha}(\bar{D})}).
		\end{aligned}
		\end{equation} Finally, we need an estimate for $v$. There holds
		$$
		v = \frac12\operatorname{tr}_A(\text{Sym}(\nabla_h w) -\tilde q)
		$$
		and consequently
		$$
		|v|_{C^{1,\alpha}(\bar{D})}\leq C (|q|_{C^{1,\alpha}(\bar{D})}+|u|_{C^{2,\alpha}(\bar{D})})
		\leq  C (|q|_{C^{1,\alpha}(\bar{D})}+|\psi|_{C^{2,\alpha}(\bar{D})}).
		$$
		In order to proceed, we use the so-called Nirenberg trick in the following way. We differentiate the first equation of (\ref{linearized equations}) twice in $\partial_1$ direction, the second equation in $\partial_1$ and $\partial_2$ direction and the third equation twice in $\partial_2$ direction. We then multiply the second equation by $(-1)$ and add all three equations. The terms involving third derivatives cancel out and we obtain
		\begin{align*}
		\partial_1\partial_1(v A_{22})&+\partial_2\partial_2(v A_{11})-2\partial_1\partial_2(v A_{12})
		\\=2\partial_1\partial_2q_{12}&-\partial_1\partial_1q_{22}-\partial_2\partial_2q_{11}+Q(\partial_i\partial_j u_m)+Q(\partial_i u_j)+Q(u_i)+Q(q_{ij})+Q(\partial_i q_{jm}).
		\end{align*}
		As before, strict convexity translates into strong ellipticity of this equation. Moreover, the Gauss-Codazzi equations imply that the second order derivatives of the second fundamental form on the left hand side cancel out. Similarly, one may check that no third order terms of the metric $h$ appear on the right hand side. On the other hand, in polar coordinates, we have 
		$$
		\partial _r v=\partial_r\Psi\cdot \nu+\Psi\cdot \partial_r\nu=\Phi\cdot \nu +E\Psi\cdot \nu +A_{rr}u_r=\Phi\cdot \nu+Ev+A_{rr}u_r,
		$$
		where we used Lemma \ref{improvedbc} and $A_{\varphi r}=0$ on $\partial D$.
		Hence the Schauder estimates with Neumann boundary conditions imply
		$$
		|v|_{C^{2,\alpha}(\bar{D})}\leq c(|w|_{C^{2,\alpha}(\bar{D})}+|q|_{C^{1,\alpha}(\bar{D})}+|2\partial_1\partial_2q_{12}-\partial_1\partial_1q_{22}-\partial_2\partial_2q_{11}|_{C^{0,\alpha}(\bar{D})}+|\Phi|_{C^{1,\alpha}(\bar{D})}+|v|_{C^{1,\alpha}(\bar{D})})
		$$
		where the derivatives are taken with respect to Euclidean coordinates. Combining this with the previous estimates, we obtain the final estimate
		$$
		|\Psi|_{C^{2,\alpha}(\bar{D})}\leq c(|q|_{C^{1,\alpha}(\bar D)}+|2\partial_1\partial_2q_{12}-\partial_1\partial_1q_{22}-\partial_2\partial_2q_{11}|_{C^{0,\alpha}(\bar{D})}+|\Phi|_{C^{1,\alpha}(\bar{D})}+|\psi|_{C^{2,\alpha}(\bar{D})}).
		$$
		The term $2\partial_1\partial_2q_{12}-\partial_1\partial_1q_{22}-\partial_2\partial_2q_{11}$ does not contain any third derivatives of $\Psi^i$ where $i< l$ and it follows that
		\begin{align*}
		|\Psi|_{C^{2,\alpha}(\bar{D})}\leq c\bigg(&
		\sum_{i=0}^l\binom{l}{i}|\partial_t^iE_{t}|_{C^{2,\alpha}(\bar{D})}|\partial_t^{l-i}E_{t}|_{C^{2,\alpha}(\bar{D})}+ \sum_{i=1}^{l-1}\binom{l}{i}|\Psi^{l-i}|_{C^{2,\alpha}(\bar{D})}\cdot |\Psi^{i}|_{C^{2,\alpha}(\bar{D})}
		\\&+\sum_{i=1}^{l}\binom{l}{i}|\partial_t^iE_{t}|_{C^{1,\alpha}(\bar{D})}|\Psi^{l-i}|_{C^{1,\alpha}(\bar{D})}
		\bigg)\bigg|_{t=0}.
		\end{align*}
	\end{proof}
	We now iteratively use this a-priori estimate to show that the power series (\ref{taylor}) converges in $C^{2,\alpha}(\bar{D})$. To this end, recall that the conformal factors of the metrics $h$ and $\tilde h$ are given by $E$ and $\tilde E$, respectively. 
	\begin{lem}
		Given $\epsilon>0$ small enough, there exists a constant $\Lambda>0$ depending only on $|F|_{C^{3,\alpha}(\bar{D})},$ $|E|_{C^{2,\alpha}(\bar{D})},$ $|\tilde E|_{C^{2,\alpha}(\bar{D})},$ $ \alpha$ and a number $\delta>0$ which additionally depends on $\epsilon$ such that the following holds. If $$|E-\tilde E|_{C^{2,\alpha}(\bar{D})}<\delta,$$ then
		$$
		|\Psi^l|_{C^{2,\alpha}(\bar{D})}\leq l! (\Lambda\epsilon)^l.
		$$ 
		
		\label{finalestimate}
	\end{lem}
	\begin{proof}
		Let us define 
		$$
		\hat{\Psi}^i=\frac{1}{i!}\Psi^i, \qquad \hat E_{i}=\frac{1}{i!} \partial_t^iE_{t}\big|_{t=0}.
		$$	
		Now, Lemma \ref{aprioriestimate} becomes
		\begin{align}
		\notag |\hat\Psi^l|_{C^{2,\alpha}(\bar{D})}\leq c\bigg(&
		\sum_{i=0}^l|\hat E_{i}|_{C^{2,\alpha}(\bar{D})}|\hat E_{l-i}|_{C^{2,\alpha}(\bar{D})}+ \sum_{i=1}^{l-1}|\hat \Psi^{l-i}|_{C^{2,\alpha}(\bar{D})} |\hat\Psi^{i}|_{C^{2,\alpha}(\bar{D})}
		\\&+\sum_{i=1}^{l}|\hat E_{i}|_{C^{1,\alpha}(\bar{D})}|\hat \Psi^{l-i}|_{C^{1,\alpha}(\bar{D})} \label{condensed apriori}
		\bigg).
		\end{align}
		We may assume that $c\geq 1$. In order to proceed, we use the following recursive estimate.
		\begin{lem}
			Let $\{y_i\}_{i=1}^\infty$ be a sequence of positive numbers, $\epsilon$, $\gamma$, $c>0$ and assume that  $$y_i\leq \epsilon^i\gamma^{i-1}c^{i-1}\prod_{j=2}^{i}\bigg(4-\frac6j\bigg)$$ for every $i<l$, where $l\in\mathbb{N}$. Then there holds
			$$
			\sum_{i=1}^{l-1} y_iy_{l-i}
			\leq \epsilon^l\gamma^{l-2}c^{l-2}\prod_{j=2}^{l}\bigg(4-\frac6j\bigg).
			$$
			\label{recursiveestimate}
		\end{lem}
		We will prove the lemma later on. We now show that for every number $l\in\mathbb{N}$, the following two estimates hold
		\begin{equation}
		|\hat \Psi^l|_{C^{2,\alpha}(\bar D)}\leq  (\Lambda\epsilon)^l, \qquad |\hat\Psi^l|_{C^{2,\alpha}(\bar{D})}\leq
		\epsilon^{l}\gamma^{l-1}c^{l-1}\prod_{j=2}^{l}\bigg(4-\frac6j\bigg).
		\label{induction argument}
		\end{equation}
		Using the explicit definition of $E_t$ from Lemma \ref{pathconnect}, we compute
		$$
		\partial_t^iE_{t}\bigg|_{t=0}=(-1)^i\frac{E}{\tilde E^{i}}(E-\tilde E)^ii!.
		$$
		Hence, given $\tilde \epsilon>0$, we can chose $\delta$ small enough such that 
		$$
		|\partial_t^i E|_{C^{2,\alpha}(\bar{D})}\big|_{t=0}\leq i!\tilde{\epsilon}^i
		$$
		and consequently
		\begin{align}
		|\hat E_{i}|_{C^{2,\alpha}(\bar{D})}\leq \tilde {\epsilon}^i, \label{1Eest}
		\end{align}
		provided
		$$|E-\tilde E|_{C^{2,\alpha}(\bar{D})}<\delta.$$
		Moreover,
		\begin{align}
		\sum_{i=0}^l|\hat E_{i}|_{C^{2,\alpha}(\bar{D})}|\hat E_{l-i}|_{C^{2,\alpha}(\bar{D})}\leq (l+1)\tilde{\epsilon}^l. \label{sumEest}
		\end{align}
		Increasing $c$ if necessary, we may arrange that $|F|_{C^{2,\alpha}(\bar{D})}\leq c$. Together with (\ref{condensed apriori}) this implies, decreasing $\tilde \epsilon$ appropriately, that
		\begin{align*}
		|\hat \Psi^1|_{C^{2,\alpha}(\bar{D})}\leq 2c\tilde \epsilon+c^2\tilde{\epsilon}<\epsilon.
		\end{align*}
		In particular, for every $\gamma\geq 1$ we have
		$$
		|\hat \Psi^1|_{C^{2,\alpha}(\bar{D})}<\gamma^{0}c^{0}\epsilon^1\prod_{j=2}^{1}\bigg(4-\frac6j\bigg).
		$$
		This proves (\ref{induction argument}) for $l=1$ and every $\Lambda,$ $\gamma>1$. \\ \indent 
		Now given $l\geq 2$, let us assume that we have already shown that
		$$
		|\hat\Psi^i|_{C^{2,\alpha}(\bar{D})}\leq \gamma^{i-1}c^{i-1}\epsilon^i\prod_{j=2}^{i}\bigg(4-\frac6j\bigg)
		$$
		for every $i<l$ and some suitable choice of $\gamma>1$.
		Then the a-priori estimate (\ref{condensed apriori}), Lemma \ref{recursiveestimate} as well as (\ref{1Eest}) and (\ref{sumEest}) imply
		\begin{align}
		|\hat\Psi^l|_{C^{2,\alpha}(\bar{D})}\leq c\bigg((l+1)\tilde\epsilon^l+  \epsilon^l\gamma^{l-2}c^{l-2}\prod_{j=2}^{l}\bigg(4-\frac6j\bigg)
		+\sum_{i=1}^{l}\tilde{\epsilon}^i
		\gamma^{l-i-1}c^{l-i-1}\epsilon^{l-i-1}\prod_{j=2}^{l-i}\bigg(4-\frac6j\bigg)
		\bigg). \label{firstineq}
		\end{align}
		If we also ensure that $\tilde{\epsilon}<\epsilon/2,$ then  
		$$
		(l+1)\tilde\epsilon^l\leq\epsilon{^l} \gamma^{l-2}c^{l-2}\prod_{j=2}^{l}\bigg(4-\frac6j\bigg).
		$$
		Furthermore, we have the trivial estimate
		\begin{align*}
		\sum_{i=1}^{l}\tilde{\epsilon}^i
		\gamma^{l-i-1}c^{l-i-1}\epsilon^{l-i-1}\prod_{j=2}^{l-i}\bigg(4-\frac6j\bigg)
		\leq \gamma^{l-2}c^{l-2}\epsilon^{l}\prod_{j=2}^{l}\bigg(4-\frac6j\bigg)\sum_{i=1}^l2^{-i}
		\leq \gamma^{l-2}c^{l-1}\epsilon^{l}\prod_{j=2}^{l}\bigg(4-\frac6j\bigg),
		\end{align*}
		provided $c\geq 2$.
		Combining this with (\ref{firstineq}), we obtain
		$$
		|\hat\Psi^l|_{C^{2,\alpha}(\bar{D})}\leq 3\gamma^{l-2}c^{l-1}\epsilon^{l}\prod_{j=2}^{l}(4-6/j)\leq
		\gamma^{l-1}c^{l-1}\epsilon^{l}\prod_{j=2}^{l}\bigg(4-\frac6j\bigg),
		$$
		provided $\gamma\geq 3$. Thus, we can choose $\Lambda=4\gamma c$ to obtain
		$$
		|\hat\Psi^l|_{C^{2,\alpha}(\bar{D})}\leq \Lambda^l\epsilon^l.
		$$
	\end{proof}
	\begin{proof}[Proof of Lemma \eqref{recursiveestimate}]
		We may assume that $\epsilon=c=\gamma=1$. We are then left to show the following identity
		$$
		\sum_{i=1}^{l-1}\prod_{j=2}^{i}\bigg(4-\frac6j\bigg)\prod_{m=2}^{l-i}\bigg(4-\frac6m\bigg)=\prod_{j=2}^{l}\bigg(4-\frac6j\bigg).
		$$ 
		There holds
		$$
		\prod_{j=2}^{l}\bigg(4-\frac6j\bigg)=\frac{2^{l-1}}{l!}\prod_{j=2}^{l}(2j-3)
		=\frac{2^{l-1}}{l!}\frac{(2(l-1))!}{2^{l-1}(l-1)!}
		=\tilde y_{l-1}
		$$	
		where $\tilde y_{j}$ is the $j$-th Catalan number. For the Catalan numbers, the well-known recurrence relation
		$$
		\tilde y_{l-1}=\sum_{i=1}^{l-1}\tilde y_{i-1}\tilde y_{l-i-1}
		$$
		holds, see for instance \cite{stanley2015catalan}. This  implies the above identity.
	\end{proof}
	We are now in the position to prove that the solution space is open.
	\begin{prop}
		Let $h,$ $\tilde h\in\mathcal \mathcal{G}^{k,\alpha}$ and $F\in C^{k+1,\alpha}(\bar{D})$ be a solution of the free boundary problem (\ref{fbp}) for $h$. Then there exists a constant ${\delta}>0$ depending only on $\alpha$, the $C^{2,\alpha}$-data of $E$ and the $C^{3,\alpha}$-data of $F$ such that the following holds. There exists a solution $\tilde F\in C^{k+1,\alpha}(\bar{D})$ of (\ref{fbp}) for $\tilde h$ provided $$|h-\tilde h|_{C^{2,\alpha}(\bar{D})}<  \delta.$$ In particular, the space $\mathcal{G}_*^{k,\alpha}$ is open with respect to the $C^{2,\alpha}$-topology.	
		\label{openness}
	\end{prop}
	\begin{proof}
		Let $E$ and $\tilde E$ be the conformal factors of $h$ and $\tilde h$, respectively, and $h_t$ the connecting analytic path from see Lemma \ref{pathconnect}. 	We solve (\ref{linearizedequations}) to obtain the maps $\Psi^l$ and define $F_t:\bar{D}\times[0,1]\to\mathbb{R}^3$ by
		$$
		F_t=F+\sum_{l=1}^\infty \frac{\Psi^l}{l!}t^l.
		$$
		First, we choose $\delta>0$ such that $|\tilde E|_{C^{2,\alpha}(\bar{D})}<2|E|_{C^{2,\alpha}(\bar{D})},$
		provided
		$|E-\tilde E|_{C^{2,\alpha}(\bar{D})}<\delta.$ Let $\Lambda$ be the constant from Lemma \ref{finalestimate} and $\epsilon =(2\Lambda)^{-1}$. Lemma \ref{finalestimate} then implies that we may decrease $\delta$ appropriately such that
		$$
		|\Psi^l|_{C^{2,\alpha}(\bar{D})}\leq l! 2^{-l}
		$$ 
		for every $l\in\mathbb{N}$
		provided
		$|E-\tilde E|_{C^{2,\alpha}(\bar{D})}< \delta$. It follows that $F_t$ converges uniformly in $C^{2,\alpha}(\bar{D})$. Now, Lemma \ref{existence criterion}  implies that $\tilde F=F_1$ is a solution of (\ref{fbp}) with respect to $\tilde h$. Standard elliptic  theory yields the claimed regularity for $\tilde F$.	
	\end{proof}

\section{Closedness of the solution space}
In this section, we will prove the closedness of the solution space $\mathcal {G}_*^{k,\alpha}$  with respect to the $C^{4}$-topology, provided $k\geq 4$. We suspect that this requirement can be weakened to $k\geq 3$ using a more refined $C^2$-estimate; cf. \cite{heinz1962weyl,schulz2006regularity}.\\ \indent  
Let $\{h_l\}_{l=1}^\infty$ be a sequence of metrics $h_l\in \mathcal {G}_*^{k,\alpha}$ converging to a metric $\tilde h\in \mathcal{G}^{k,\alpha}$ in $C^{4}$. We use a compactness argument to show that $\tilde h\in\mathcal {G}_*^{k,\alpha}$. To this end, we prove an a-priori estimate for $|F_l|_{C^{2,\alpha}(\bar{D})}$, where $F_l$ is the solution (\ref{fbp}) with respect to $h_l$. 

\subsection{A global curvature estimate}
We fix $l\in\mathbb{N}$ and abbreviate $F=F_l$ as well as $h=h_l$. Since $F$ is an isometric embedding of a compact disc, it is bounded in $C^1$ in terms of the $C^0$-norm of the metric $h$, that is,
$$
|F|_{C^{1}(\bar{D})}\leq c|h|_{C^0(\bar{ D})}. 
$$
Furthermore, in every coordinate chart, there holds
$$
\partial_i\partial_j F =\Gamma^m_{ij}\partial_m F -A_{ij}\nu.
$$   
The second fundamental form can be expressed in terms of the Gauss curvature and the mean curvature. In fact, if $\kappa_1,\kappa_2$ are principal directions, then
$$
\kappa_1=\frac{H}{2}+\sqrt{\frac{H^2}{4}-K},\qquad \kappa_2=\frac{H}{2}-\sqrt{\frac{H^2}{4}-K}.
$$ 
This implies the following estimate.
\begin{lem}
	Let $h\in\mathcal{G}_*^{k,\alpha}$ and $F$ be a solution of (\ref{fbp}) with respect to $h$. Then there is a constant $c$ which only depends on $|h|_{C^1(\bar{ D})}$ and $|H|_{C^0(\bar{ D})}$ such that
	\label{closelem1}
	$$
	|F|_{C^2(\bar{ D})}\leq c.
	$$	

\end{lem}
In order to bound $H$ in terms of the intrinsic geometry, we  use the maximum principle. To this end, we observe that the free boundary condition implies a useful formula for the normal derivative of the mean curvature at the boundary $\partial D$.
\begin{lem}
	Let $h\in\mathcal{G}_*^{k,\alpha}$ and let $F$ be a solution of (\ref{fbp}) with respect to $h$. Then there exists a constant $c$ which only depends on $|h|_{C^{4}(\bar{ D})}$ and $|1/K|_{C^0(\bar{ D})}$  such that 
	\label{clselem2}
	\begin{align*}
	|H|_{C^0(\bar{D})}\leq c.
	\end{align*}
\end{lem}
\begin{proof}
	We choose Fermi-coordinates $h=\zeta^2ds^2+dt^2$ adapted to the boundary $\partial D$ with $\zeta(s,0)=1$. Since the geodesic curvature is equal to $1$, it follows that $\partial_t \zeta(\cdot,0)=-1$. One may compute that the only non-zero Christoffel symbols at the boundary are
	$$
	\Gamma^{s}_{st}=-1, \qquad \Gamma^t_{ss}=1.
	$$
	In this coordinate chart, the mean curvature is given by
	$$
	H=\zeta^{-2}A_{ss}+A_{tt}.
	$$
	Since $A_{st}$ vanishes on $\partial D$, see (\ref{A cond boundary}), the Gauss-Codazzi equations imply
	$$
	\partial_tA_{ss}=-H.
	$$
	Consequently, differentiating the Gauss equation and using $\partial_t \zeta=-1$, we find
	\begin{align*}
	\partial_t A_{tt}A_{ss}= \partial_t A_{ss}A_{tt}+\partial_t A_{tt}A_{ss}-2\partial_t A_{st}A_{st}+HA_{tt}=\partial_t (K\zeta^2)+HA_{tt}=\partial_t K-2K+HA_{tt}.
	\end{align*}
	on $\partial D$, where we used that  $A_{st}=0 $. In particular,
	$$
	\partial_tA_{tt}=\frac{\partial_t K}{K} A_{tt}- A_{tt}+\frac{A^2_{tt}}{A_{ss}}
	$$
	on $\partial D$.	 It follows that
	\begin{align}
	\partial_t H&\geq A_{ss}-\Lambda A_{tt}+\frac{A^2_{tt}}{A_{ss}}\label{boundaryestee},
	\end{align}
	where
	$$
	\Lambda=\operatorname{max}_{p\in D}|K^{-1}\nabla K|+2.
	$$
	Now,  let $\gamma>0$ be such that $A_{11}=\gamma A_{22}$. It follows that 
	$$
	\partial_t H>0
	$$
	unless 
	\begin{align}
	\Lambda^{-1}\leq\gamma\leq\Lambda. \label{principal comparison}
	\end{align}
	Now, suppose that $H$ attains its global maximum at $p\in\bar{D}$. If $p\in\partial D$, then
	$$
	\partial_t H\leq 0
	$$  
	and it follows that (\ref{principal comparison}) holds. We may assume that $A_{ss}\leq A_{tt}$ and estimate
	$$
	H\leq 2A_{tt}\leq 2\sqrt{\Lambda}\sqrt{A_{ss}A_{tt}}\leq 2\sqrt{\Lambda}\sqrt{K},
	$$
	as claimed. So let us assume that $p\in D$. We choose normal coordinates $\partial_1$, $\partial_2$ centred at $p$. Clearly, we have
	\begin{align}
	0&=\partial_1 A_{11}+\partial_1 A_{22},
	\label{firstmax} \\
	0&=\partial_2 A_{11}+\partial_2 A_{22}.
	\label{secondmax}
	\end{align}
	As has been shown in Lemma 9.3.3. in \cite{han2006isometric}, there also holds 
	\begin{align*}
	A^{ij}\partial_{ij}H=-\frac{2}{K}(\partial_lA_{11}\partial_lA_{22}-\partial_l A_{12}\partial_l A_{12}) +H^2-4K+\frac{\Delta_hK}{K}.
	\end{align*}
	Since the Hessian of $H$ is non-positive at $p$, strict convexity, (\ref{firstmax}) and (\ref{secondmax}) imply that
	\begin{align}
	0\geq&-\frac{2}{K}(\partial_lA_{11}\partial_lA_{22}-\partial_l A_{12}\partial_l A_{12}) +H^2-4K+\frac{\Delta_hK}{K}  \notag\\
	&=\frac{2}{K}(\partial_lA_{11}\partial_lA_{11}+\partial_l A_{12}\partial_l A_{12}) +H^2-4K+\frac{\Delta_hK}{K}  \notag\\
	&\geq H^2-4K+\frac{\Delta_h K}{K}.
	\end{align}
	The claim follows.
\end{proof}
\subsection{A Krylov-Evans type estimate}
Now, we improve the $C^2$-estimate to a $C^{2,\alpha}$-estimate.  The potential function
$$
f=\frac12 |F|^2
$$
can be used to estimate the second fundamental form $A$; cf. \cite{han2006isometric}. Namely, there holds
$$
\partial_i\partial_j f=\Gamma^m_{i,j}\partial_m f -A_{ij}F\cdot \nu +h_{ij}.
$$
It follows  that $|A|_{C^{0,\alpha}(\bar{D})}$ can be estimated in terms of $|h|_{C^{1,\alpha}(\bar{ D})}$ and $|f|_{C^{2,\alpha}(\bar{ D})}$ provided that $|F\cdot \nu|$ is uniformly bounded from below. Taking the determinant of both sides of the equation, we find that   $f$ satisfies the following Monge-Ampere type equation
\begin{align}
\begin{dcases}
&\det(\partial_{ij}f-\Gamma^m_{ij}\partial_mf-h_{ij})=\det(h)K(F\cdot \nu)^2
\text{ in }D, \\
&\partial_\mu f = 1 \text{ on }\partial D.
\end{dcases}
\label{MAP}
\end{align}
This suggests that a Krylov-Evans type estimate might be applicable. The major obstacle in this regard is that the free boundary condition implies that $F\cdot \nu=0 $ on $\partial D$. 
\begin{lem}
	Let $h\in{\mathcal G}_*^{k,\alpha}$ and  $F$ be a solution of the free boundary problem (\ref{fbp}) with respect to $h$. Then there exists a constant $c$ which only depends on $ |h|_{C^{4}(\bar{ D})},|1/K|_{C^0(\bar{ D})}$ and $\alpha\in(0,1)$ such that \label{close3est}
	\begin{align*}
	|F|_{C^{2,\alpha}(\bar{D})}\leq c.
	\end{align*}

\end{lem}
\begin{proof}
	Let $(\varphi,r)$ be isothermal polar coordinates centred at the origin with conformal factor $E$.  We consider the Gauss map $\nu:\bar{D}\to \mathbb{S}^2\subset\mathbb{R}^3$. Since $F$ is strictly convex,  $\nu$ is a strictly convex embedding, too. Let $\hat h=\nu^*\bar g$ be the pull-back metric of the Euclidean metric $\bar g$. There holds $A_{r\varphi}=0$ on $\partial D$; cf. (\ref{A cond boundary}). This implies that $$\partial_\varphi \nu=E^{-2}A_{\varphi\varphi}\partial_\varphi F, \qquad  \partial_r \nu=E^{-2}A_{rr} \partial_r F,$$ on $\partial D$. It follows that $\hat\mu=E^{-1}\partial_r F=\mu$ is the outward co-normal of $\partial D$ with respect to $\hat h$. Consequently, the geodesic curvature of $\hat h$ along $\partial D$ is given by
	$$
	k_{\hat h}=\frac{E^3}{ A^2_{\varphi\varphi}} \partial_\varphi\partial_\varphi \nu\cdot \partial_r F=\frac{E^2}{A_{\varphi\varphi}} \partial_\varphi\partial_\varphi F\cdot 
	\mu=\frac{1}{A_{\varphi\varphi}},
	$$
	where we used the fact that $k_h=1$. Thus, the previous lemma implies that $$k_{\hat{h}}\geq \eta>0$$ for some constant $\eta>0$ which can be uniformly bounded from below in terms of $|h|_{C^{4}(\bar{ D})}$ and $|1/K|_{C^0(\bar{ D})}$. Arguing as in the proof of Lemma \ref{existlem}, we may assume, after a suitable rotation, that $\nu(\bar{ D})$ is contained in the lower hemisphere $\mathbb{S}^2_-$ and that the function  $\nu\cdot e_3$ attains its  maximum $s<0$ in at least two points $p_1,p_2\in\nu(\partial D)$. It then follows that there has to be another point $p\in\nu(\partial D)$ where $k_{\hat h}\leq k_{s}=-s\sqrt{1-s^2}$.  Here, $k_{s}$ is the curvature of the curve $\{ p\in\mathbb{S}^2:p\cdot e_3=s\}$. Consequently\footnote{A similar result is also due to W. Fenchel, see \cite{fenchel1929krummung}.}, $$\nu\cdot e_3\leq  s\leq- \frac12 \eta.$$ As we have seen in the proof of Lemma (\ref{existlem}), $F\cdot \nu$  vanishes precisely at the boundary. This implies that  $ F\cdot \nu\leq 0$ 
	and it follows that $\hat F\cdot \nu\leq s<0$, where $\hat F=F+e_3$. Clearly, $\hat F$ satisfies estimates comparable  to $F$. We define the function $\hat{f}=\frac12 \hat F\cdot \hat F$ and obtain
	\begin{align}
	&\det(\partial_{ij}\hat f-\Gamma^m_{ij}\partial_m\hat f-h_{ij})=\det hK(\hat F\cdot \nu)^2
	\text{ on }D, \\
	&\partial_\mu \hat f = 1-F\cdot e_3 \text{ on }\partial D.
	\end{align}
	Since $\hat F\cdot \nu\leq s<0$, the equation is uniformly elliptic and the ellipticity constant can be estimated in terms of $|h|_{C^{4}(\bar{ D})}$ and $|1/K|_{C^0(\bar{ D})}$. Now, the Krylov-Evans type estimate \cite[Theorem 6]{trudinger1984boundary} implies
	$$
	|\hat f|_{C^{2,\alpha}(\bar{D})}\leq  c,
	$$
	where $c$ depends on $\alpha$, $|h|_{C^{4,\alpha}(\bar{ D})}$, the ellipticity constant, $|f|_{C^2(\bar{ D})}$ and the boundary data. All these terms can be estimated in terms of $|h|_{C^{4}(\bar{ D})}$ and $|1/K|_{C^0(\bar{ D})}$.
\end{proof}
We  now prove the main result of this section.
\begin{prop}
	Let $\{h_l\}_{l=1}^\infty$ be a sequence of Riemannian metrics $h_l\in \mathcal{G}_*^{k,\alpha}$ converging in $C^{4}$ to a Riemannian metric $\tilde h\in \mathcal{G}^{k,\alpha}$, where $k\geq 4$. Then $\tilde h\in \mathcal{G}_*^{k,\alpha}$.
	\label{closeness}
\end{prop}
\begin{proof}
	Let $K_l$ denote the curvature of the metric $h_l$. The convergence implies that there is a number $\Lambda$ such that
	$ |h_l|_{C^{4}(\bar{ D})}$, $|1/K_l|_{C^0(\bar{ D})}\leq\Lambda$ for all $l\in\mathbb{N}$.	Lemma \ref{close3est} then implies the uniform estimate
	$$
	|F_l|_{C^{2,\alpha}(\bar{ D})}\leq c,
	$$
	where $F_l$ are the respective solutions of (\ref{fbp}). According to the Arzela-Ascoli theorem, we can extract a subsequence converging in $C^2(\bar{D})$ to a map $F\in C^2(\bar{D})$ which is a solution of the free boundary problem (\ref{fbp}) with respect to $\tilde h$. Standard elliptic theory implies the claimed regularity.
\end{proof}
\section{Proof of Theorem \ref{main theorem fbie}}
\subsection{Existence.} It suffices to show that $\mathcal{G}_*^{k,\alpha}=\mathcal{G}^{k,\alpha}$.  According to Lemma \ref{pathconnect}, ${\mathcal{G}}^{k,\alpha}$ is path-connected while  Proposition \ref{openness} and Proposition \ref{closeness} imply that $\mathcal{G}_*^{k,\alpha}$ is open and closed. We define the map 
	$$
	F_0:[0,4\pi]\times [0,2\pi]\qquad \text{given by } \qquad F(\theta,\varphi)=(\sin\theta\sin\varphi,\sin\theta\cos\varphi,\sqrt{2}-\cos\theta).
	$$	
	The image of $F_0$ has positive curvature, geodesic curvature along the boundary equal to $1$ and meets the unit sphere orthogonally. Consequently, $\mathcal{G}_*^{k,\alpha}$ is a non-empty, open and closed subset of a path-connected space which  implies that $\mathcal{G}_*^{k,\alpha}=\mathcal{G}^{k,\alpha}$.
	\subsection{Uniqueness.}
	Suppose that $\tilde F$ is another solution of (\ref{fbp}). We denote the Gauss curvature of $h$ by $K$, and the second fundamental forms of $F$ and $\tilde F$ by $A=(A_{ij})_{ij}$ and $\tilde A=(\tilde A_{ij})_{ij}$, respectively. The respective normals $\nu$ and $\tilde \nu$ are chosen in a way such that the respective mean curvatures, denoted by $H$ and $\tilde H$, share the same sign. We now use a variation of the argument in \cite{pigola2003some}. \\ \indent  After a rotation and reflection, we may assume that both $F$ and $\tilde F$ are contained in the open upper hemisphere, that $\nu$ and $\tilde \nu$ both point downwards and that $F\cdot \nu$ and $\tilde F \cdot \tilde \nu$ vanish precisely at the boundary. Next, we choose a local orthonormal frame $e_1,e_2$ with respect to $h$ and  define the following two vector fields
	$$
	X=\bigg(\tilde A_{ij}F\cdot dF(e_j) -\tilde HF\cdot dF(e_i)\bigg)e_i, \qquad
	Y=\bigg(A_{ij}F\cdot dF(e_j)-HF\cdot dF(e_i)\bigg)e_i.
	$$
	It can be checked that these definitions do not depend on the choice of the orthonormal frame. Using the conformal property of the position vector field in $\mathbb{R}^3$, one computes as in  \cite[Proposition 1.2 and Proposition 1.8]{pigola2003some}) that
	\begin{align}
	\operatorname{div}_h X&=-\tilde H-2KF\cdot \nu+F\cdot \nu \det(A-\tilde A), \label{1div}
	\\\operatorname{div}_h Y&=-H-2KF\cdot \nu . \label{2div}
	\end{align}
	Let $(\varphi,r)$ denote isothermal polar coordinates, $\mu$ the outward co-normal of the Riemannian manifold $(D,h)$ and $E$ the conformal factor of $h$. Using the free boundary condition and $A_{r\varphi}=\tilde A_{r\varphi}=0$, we find
	\begin{align*}
	X\cdot \mu=-E^{-2}\tilde A_{\varphi\varphi} \qquad 
	Y\cdot \mu =-E^{-2} A_{\varphi\varphi}
	\end{align*}
	on $\partial D$.
	Hence, integrating (\ref{1div}) and (\ref{2div}) over $M$, applying the divergence theorem and subtracting both equations we find that
	$$
	\int_{\partial D} E^{-2} (A_{\varphi\varphi}-\tilde A_{\varphi\varphi})\,\text{d}vol_h= \int_D (H-\tilde H)\, \text{d}vol_h +\int_D F\cdot \nu\det(A-\tilde A)\,\text{d}vol_h.
	$$
	Interchanging the roles of $F$ and $\tilde F$ and performing the same computation again, we conclude
	$$
	\int_D (F\cdot \nu+\tilde F \cdot \tilde \nu) \det(A-\tilde A)\text{d}vol_h=0.
	$$
	$F\cdot \nu,$ $\tilde F \cdot \tilde \nu$ are both positive on $D$. Conversely, $A,\tilde A>0$ and 
	$$
	\det(A)=\det(\tilde A)=\det(h)K
	$$
	imply that $\det(A-\tilde A)\geq 0$. It follows that
	$$
	\det(A-\tilde A)=0
	$$
	and thus
	$$
	A=\tilde A.
	$$
	Hence, $F$ and $\tilde F$ share the same second fundamental form and consequently only differ by a rigid motion. Since both of their boundaries are contained in the unit sphere, this rigid motion must be a composition of a rotation and a reflection through a plane containing the origin.

\appendix
\section{The Lopatinski-Shapiro condition}
In this section, we verify the Lopatinski-Shapiro condition. We refer to \cite{wendlandelliptic} for its definition.
\begin{lem}
	$\mathcal L$ satisfies the Lopatinski-Shapiro condition.
\end{lem}
\begin{proof}
	We interpret $L$ as an operator from $\mathcal{C}^{1,\alpha}$ to $\mathcal{C}^{0,\alpha}$ and let $\partial_1, \partial_2$ be any local coordinate system on $D$. We compute  $L(u)=\mathcal{M}_1\partial_1 u +\mathcal{M}_2\partial_2(u)+\mathcal{M}_3u$, where
	\begin{align*}
	\mathcal{M}_1=\begin{pmatrix}
	1-\frac12A^{11}A_{11} & -\frac12 A^{12}A_{11}\\
	-\frac12 A^{11}A_{12} & 1 - \frac12 A^{12}A_{12}
	\end{pmatrix}
	\qquad  \mathcal{M}_2=\begin{pmatrix}
	-\frac12 A^{12}A_{11} & -\frac12 A^{22}A_{11}\\
	1-\frac12 A^{12}A_{12} &  - \frac12 A^{22}A_{12}
	\end{pmatrix}
	\end{align*}	
	and $\mathcal{M}_3$ is a matrix which we do not need to determine. Calculating the inverse of $A$ we find 
	\begin{align*}
	\mathcal{M}_1&=\begin{pmatrix}
	1-(2\det(A))^{-1}A_{11}A_{22} & (2\det(A))^{-1} A_{12}A_{11}\\
	-(2\det(A))^{-1} A_{22}A_{12} & 1 + (2\det(A))^{-1} A_{12}A_{12}
	\end{pmatrix}, \\
	\qquad \mathcal{M}_2&=\begin{pmatrix}
	(2\det(A))^{-1}A_{12}A_{11} & -(2\det(A))^{-1} A_{11}A_{11}\\
	1+(2\det(A))^{-1}A_{12}A_{12} &  - (2\det(A))^{-1}A_{11}A_{12}
	\end{pmatrix}.
	\end{align*}	
	Let $p\in\partial D$ and choose polar coordinates $(\varphi,r)$ centred at the origin with $\partial_1=\partial_\varphi$ and $\partial_2=\partial_r$. The free boundary condition implies $A_{12}=0$, see (\ref{A cond boundary}). Furthermore, there holds $\det(A)=A_{11}A_{22}$. Consequently,
	\begin{align*}
	\mathcal{M}=\mathcal{M}_1^{-1}\mathcal M_2=\begin{pmatrix}
	0 &-\psi^2 \\
	1 & 0
	\end{pmatrix}
	\end{align*}
	where $\psi^2=\frac{A_{11}}{A_{22}}$. The eigenvalues of this matrix are given by $\pm i\psi$. In the chosen coordinate chart, the boundary operator $R_1$ is given by  $$R_1=\begin{pmatrix}0\hspace{0.3cm} & 1\end{pmatrix}.$$ According to \cite[16.1]{wloka1995boundary}, the Lopatinski condition is satisfied if the matrix
	\begin{align*}
	R_1 \int_{\gamma}(\xi \operatorname{Id}-\mathcal{M})^{-1}\text{d}\xi
	\end{align*}
	has rank $1$ for every closed path $\gamma$ in the upper half plane containing $i\psi$.
	We compute
	\begin{align*}
	(\xi \operatorname{Id}-\mathcal M)=\frac{1}{(\xi-i\psi)(\xi+i\psi)}\begin{pmatrix} \xi & \psi^2 \\  -1 & \xi \end{pmatrix}.
	\end{align*}
	The residue theorem implies
	\begin{align*}
	R_1 \int_{\gamma}(\xi \operatorname{Id}-\mathcal{M})^{-1}\text{d}\xi=\pi \begin{pmatrix}
	-\psi^{-1} \hspace{0.3cm}& i
	\end{pmatrix}.
	\end{align*}
	\label{lopatinski condition}
\end{proof}

\section{Explicit examples in the Schwarzschild space}
\label{schwarzschild}
In this section, we consider explicit free boundary surfaces supported on the spheres of symmetry in the Schwarzschild space and provide some numerical evidence for the validity of Conjecture \ref{conjecture}. 
\\ \indent On $\mathbb{R}^3\setminus\{0\}$, we consider the Schwarzschild metric $g_m$ with mass $m>0$ defined by
$$
g_m=\bigg(1+\frac{m}{2|x|}\bigg)^4\bar g=\phi^4_m(|x|)\bar g,
$$
where $\bar g$ denotes the Euclidean metric and $\phi_m$ the conformal factor of $g_m$. The continuous function
$$
\psi:[m/2,\infty)\to[0,\infty),\qquad \lambda \mapsto H(S_\lambda(0))=2 \phi_m^{-2}(\lambda)\lambda^{-1}-2m\phi_m^{-3}(\lambda) \lambda^{-2}.
$$
vanishes at $\lambda=m/2$ and approaches $0$ as $\lambda \to\infty$.  A lengthy but straightforward calculation shows that
$$
\operatorname{max}\psi_m > 2 \qquad  \text{ if and only if }\qquad m<{3^{-\frac32}}.
$$
Let $m\in(0,3^{-\frac32})$ and $\lambda_m>m/2$ be maximal such that
$$
H(S_{\lambda_m}(0))=2.
$$

\begin{figure}
	\includegraphics[width=0.5\linewidth]{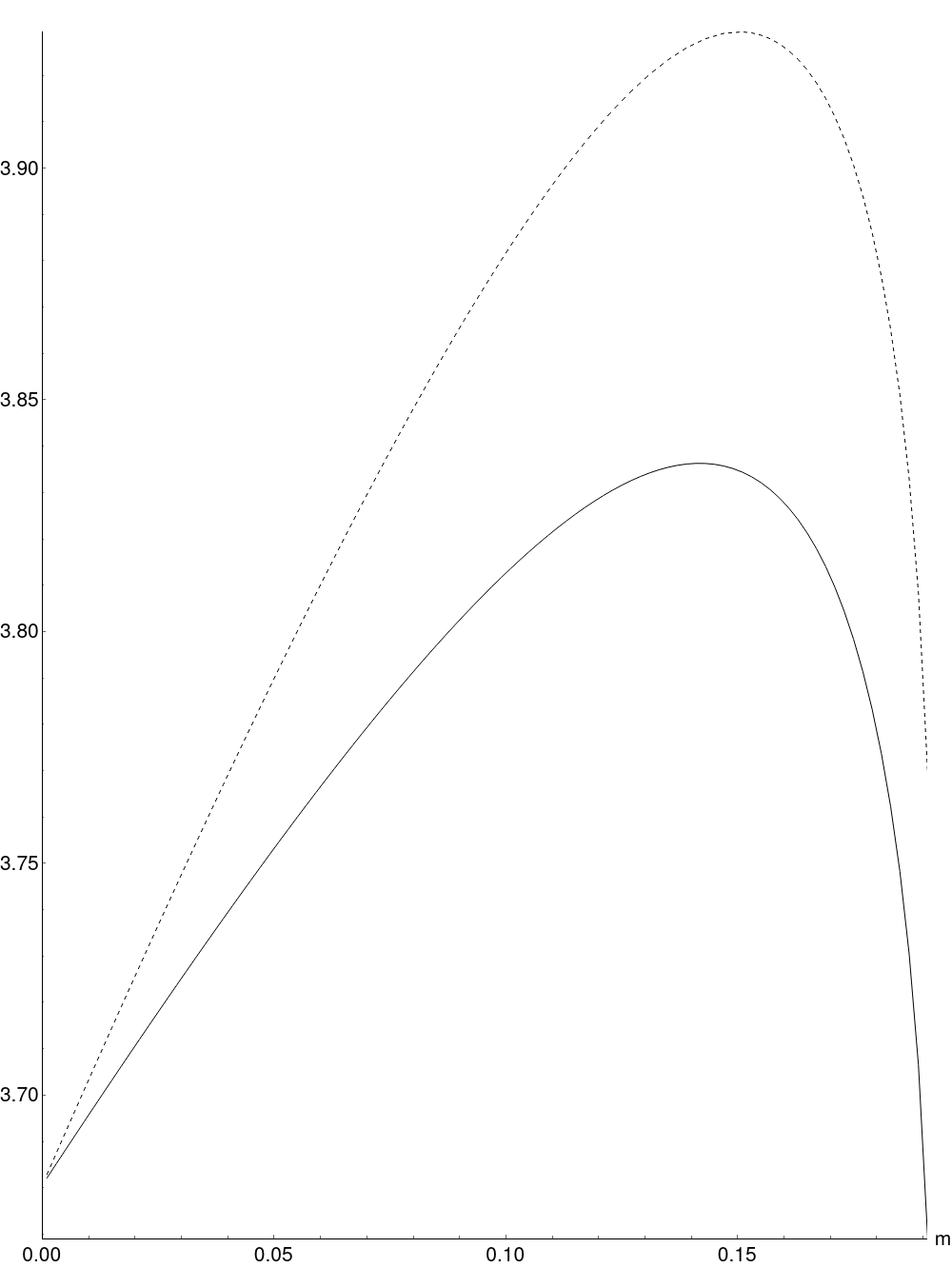}
	\label{figure1}
	\caption{A plot of the total mean curvature of the free boundary surface $\Sigma$ with respect to the Euclidean background metric and the Schwarzschild background metric for different values of $m$. The dotted line, corresponding to the total Euclidean mean curvature, lies above the solid line, corresponding to the Schwarzschild background metric.}
\end{figure}
We consider the embedded disc $\Sigma=\operatorname{Im}(\Phi)$, where $$\Phi:[0,\pi/4]\times[0,2\pi)\to \mathbb{R}^3\setminus\{0\} \qquad (\theta,\varphi)\mapsto \lambda_m(\sin\theta\sin\varphi,\sin\theta\cos\varphi,\sqrt{2}-\cos\theta).$$
Let $h$ be the metric induced by the Schwarzschild background metric. One may check  that the geodesic curvature of $\partial \Sigma$ satisfies
$$
k_h=\frac12 H(S_{\lambda_m}(0))=1.
$$
Similarly, the Gauss equation shows that
$$
K_h=\frac12 H^2-\operatorname{Rc}(\nu,\nu)>0.
$$
Here, we recall that 
$$
H=2\phi_m^{-2}(|x|)\lambda_m^{-1}-2m\phi_m^{-1}|x|^{-3}x\cdot \nu, \qquad  
\operatorname{Rc}(\nu,\nu)=2m \phi_m^{-6}(|x|) |x|^{-3}(1-3\phi_m^4(|x|) |x|^{-2} (x\cdot  \nu)^2 ).
$$
Let $M$ be the closure of the component of $B_{\lambda_m}(0)\setminus \Sigma$ with less volume. Since the scalar curvature of the Schwarzschild manifold vanishes,  $M$ satisfies the assumptions of Conjecture \ref{conjecture}.
\\ \indent In order to check if Conjecture \ref{conjecture} holds up to these examples, we explicitly solve the isometric embedding problem. To this end, we first observe that
$$
|\Phi|^2=\lambda_m^2(3-2\sqrt{2}\cos\theta).
$$
Consequently, there exists a function $\eta:[0,\pi/4]\to (0,\infty)$ such that
$$
\eta(\theta)=\phi_m(|\Phi(\theta,\varphi)|)
$$
for every $\varphi\in(0,2\pi)$. Moreover, we notice that
$$
h=\lambda_m^2\eta^4(\theta)(\sin^2(\theta)\text{d}\varphi^2+\text{d}\theta^2).
$$
We then consider the map
$$\Psi:[0,\pi/4]\times[0,2\pi)\to \mathbb{R}^3, \qquad (\theta,\varphi)\mapsto \lambda_m(y(\theta)\sin\varphi,y(\theta)\cos\varphi,z(\theta)).$$
The isometric embedding equation becomes
\begin{equation}
\begin{aligned}
&y(\theta)=\eta^2(\theta)\sin\theta, \\
&(y'(\theta))^2+(z'(\theta)^2)=\eta^4(\theta).
\end{aligned}
\end{equation}
It can be checked that, as predicted by Theorem \ref{main theorem fbie}, the solution to this system can be chosen such that $\Psi(\Sigma)$ meets $\mathbb{S}^2$ orthogonally along $\Psi(\partial \Sigma)$. Let $\bar \nu$ and $\bar H$ be the normal and mean curvature of $\Psi(\Sigma)\subset \mathbb{R}^3$. A direct computation gives
$$
\bar\nu =(y'(\theta))^2+(z'(\theta)^2)^{-\frac12}(z'(\theta)\sin\phi,z'(\theta)\cos\phi,-y'(\theta)),\qquad \text{d}vol_h=\lambda_m^2 y(\theta)\sqrt{y'(\theta))^2+(z'(\theta)^2},
$$
as well as
$$
\bar H= \lambda_m^{-1}(y'(\theta)^2+z'(\theta)^2)^{-\frac32}\big[z''(\theta)y'(\theta)-y''(\theta)z'(\theta)\big]+\lambda_m^{-1}(y'(\theta)^2+z'(\theta)^2)^{-\frac12}y(\theta)^{-1}z'(\theta).
$$
Conversely, we have
$$
H=2\lambda_m^{-1}\eta^{-2}(\theta)+2m\lambda_m^{-2}\eta^{-3}(\theta)(3-2\sqrt2\cos\theta)^{-\frac32}(\sqrt{2}\cos\theta-1).
$$
A numerical computation for several sample values of $m$ suggests that
$$
\int_{\Sigma}\bar H\,\text{d}vol_h>\int_\Sigma H\,\text{d}vol_h
$$
for every $m\in(0,3^{-3/2})$, see Figure 1.

\end{document}